\documentclass[11pt]{article}
\usepackage[T1]{fontenc}
\usepackage{latexsym,amssymb,amsmath,amsfonts,amsthm}
\usepackage{graphics}
\usepackage{graphicx}
\usepackage{mathrsfs}
\usepackage{subfigure}
\usepackage{float}
\usepackage{array}
\usepackage{multirow}
\newcommand{\R}{{\mat R}}
\newcommand{\Z}{{\mat Z}}
\newcommand{\N}{{\mat N}}

\newcommand{\ds}{\displaystyle}
\newcommand{\no}{\nonumber}
\newcommand{\be}{\begin{eqnarray}}
\newcommand{\ben}{\begin{eqnarray*}}
\newcommand{\en}{\end{eqnarray}}
\newcommand{\enn}{\end{eqnarray*}}

\newcommand{\pa}{\partial}

\newcommand{\ov}{\overline}
\newcommand{\I}{{\rm Im}}
\newcommand{\Rt}{{\rm Re}}

\newcommand{\G}{\Gamma}

\newcommand{\vep}{\varepsilon}

\newcommand{\al}{\alpha}

\newcommand{\wid}{\widetilde}

\newcommand{\mat}{\mathbb}

\newtheorem{theorem}{Theorem}[section]

\newtheorem{remark}[theorem]{Remark}

\newtheorem{algorithm}{Algorithm}[section]
\graphicspath{{figs/}}
\usepackage{color}
\newcommand{\high}[1] {{\color{black}{#1}}}

\begin{document}
\renewcommand{\theequation}{\arabic{section}.\arabic{equation}}

\title{\bf An FFT-based algorithm for efficient computation of Green's
functions for the Helmholtz and Maxwell's equations in periodic domains
}
\author{Bo Zhang\thanks{LSEC and Institute of Applied Mathematics, AMSS, Chinese Academy of Sciences,
Beijing, 100190, China and School of Mathematical Sciences, University of Chinese Academy of Sciences,
Beijing 100049, China ({\tt b.zhang@amt.ac.cn})}
\and
Ruming Zhang\thanks{Center for Industrial Mathematics, University of Bremen, Bremen 28359, Germany
({\sf rzhang@uni-bremen.de}); corresponding author}
}
\date{}
\maketitle

\begin{abstract}
The integral equation method is widely used in numerical simulations of 2D/3D acoustic and electromagnetic 
scattering problems, which needs a large number of values of the Green's functions. A significant topic is 
the scattering problems in periodic domains, where the corresponding Green's functions are quasi-periodic.
The quasi-periodic Green's functions are defined by series that converge too slowly to be used for calculations. 
Many mathematicians have developed several efficient numerical methods to calculate quasi-periodic Green's functions.
In this paper, we will propose a new FFT-based fast algorithm to compute the 2D/3D quasi-periodic Green's 
functions for both the Helmholtz equations and Maxwell's equations. The convergence results and error estimates 
are also investigated in this paper. Further, the numerical examples are given to show that, when a large
number of values are needed, the new algorithm is very competitive.

\vspace{.2in}
{\bf Keywords: FFT-based algorithm, periodic domain, Green's function, interpolation, convergence}
\end{abstract}

\section{Introduction}\label{sec1}
\setcounter{equation}{0}

In the scattering theory, the integral equation method is an efficient and widely used method for
numerical simulations. During the numerical procedure to solve the integral equations, a large number 
of values of the Green's functions and their derivatives are needed. In recent years, many mathematicians 
are interested in the topic that the acoustic/electromagnetic fields scattered in periodic domains with 
quasi-periodic incident waves, where the corresponding Green's functions are quasi-periodic. In this case, 
the quasi-periodic Green's functions, that are defined by slowly convergent series, are always
very difficult to evaluate. In this paper, we will develop a new FFT-based method to calculate 
the Green's functions in 2D periodic domains or 3D doubly periodic domains for the Helmholtz equations 
and Maxwell's equations numerically.

1) The quasi-periodic Green's function $G(x)$ for Helmholtz equations in 2D domains.

Assume that the 2D domain is $2\pi$-periodic in $x_1$-direction, and the incident wave $u^i$
is $\alpha$-quasi-periodic, i.e., $u^i(x_1+2\pi, x_2)=e^{i2\pi\alpha }u^i(x_1,x_2)$. The corresponding
Green's function is also $\alpha$-quasi-periodic and can be defined by the basic image expansion
\be\label{ii}
G(x)=\frac{i}{4}H_0^{(1)}(k|x|)+\frac{i}{4}\sum_{n\neq 0,\, n\in\Z}e^{i2\pi\alpha n}H_0^{(1)}(kr_n),
\en
where $r_n=\sqrt{(x_1-2\pi n)^2+x_2^2}$. It can also be defined by the basic eigenfunction expansion \cite{CML1}:
\be\label{ei}
G(x)=\frac{i}{4\pi}\sum_{n=-\infty}^\infty\frac{1}{\beta_n}\exp({i\alpha_n x_1+i\beta_n |x_2|}),
\en
where $\alpha_n=\alpha+n$,
\begin{equation*}
\beta_n=\begin{cases}
\sqrt{k^2-\alpha_n^2},
\quad &\text{if $|\alpha_n|\le k$,}\\
i\sqrt{\alpha_n^2-k^2},\quad &\text{otherwise.}
\end{cases}
\end{equation*}

2) The doubly quasi-periodic Green's function $G_d(x)$ for Helmholtz equations in 3D domains.

Assume that the 3D domain is $2\pi$-periodic in both $x_1$- and $x_2$-directions, and the incident wave $u^i$
is doubly quasi-periodic, i.e., for any integers $m_1,\,m_2$, $u^i(x_1+2m_1\pi, x_2+2m_2\pi,x_3)=e^{i2m_1\pi\alpha_1}e^{i2m_2\pi\alpha_2}u^i(x_1,x_2,x_3)$. Then the corresponding Green's 
function is also doubly quasi-periodic and can be defined by the basic eigenfunction expansion \cite{CML2}:
\be\label{he}
G_d(x)=\frac{i}{8\pi^2}\sum_{n_1,n_2\in\mathbb{Z}}\frac{1}{\beta_{n_1,n_2}}
\exp(i\alpha_{1,n_1}x_1+i\alpha_{2,n_2}x_2+i\beta_{n_1,n_2}|x_3|),
\en
where $\alpha_1,\,\alpha_2\in\mathbb{R}$, $\alpha_{1,n_1}=\alpha_1+n_1$, $\alpha_{2,n_2}=\alpha_2+n_2,$
\begin{equation*}
\beta_{n_1,n_2}=\begin{cases}
\sqrt{k^2-\alpha_{1,n_1}^2-\alpha_{2,n_2}^2},
\quad &\text{if $\alpha_{1,n_1}^2+\alpha_{2,n_2}^2\leqslant k^2$,}\\
i\sqrt{\alpha_{1,n_1}^2+\alpha_{2,n_2}^2-k^2},\quad &\text{otherwise.}
\end{cases}
\end{equation*}

3) The doubly quasi-periodic Green's tensor $\mathbb{G}(x)$ for Maxwell's equations in 3D domains.

The doubly quasi-periodic Green's tensor $\mathbb{G}(x)$ for Maxwell equations is a $3\times 3$
matrix defined by $G_d(x)$:
\be\label{me}
\mathbb{G}(x)=G_d(x)\mathbb{I}+\frac{1}{k^2}\nabla\nabla\cdot(G_d(x)\mathbb{I}),
\en
where $\mathbb{I}$ is the $3\times 3$ identity matrix. $\mathbb{G}(x)$ is
composed of $G_d(x)$ and its second order derivatives. So the numerical method of $\mathbb{G}$
comes directly from that of $G_d(x)$.

In this paper, we will consider numerical methods to calculate the functions $G(x)$, $G_d(x)$ and
$\mathbb{G}(x)$ in 2D/3D domains. We always assume that the cases we considered are away from
Wood anomalies, i.e., $\beta_n\neq 0$ for the 2D case and $\beta_{n_1,n_2}\neq 0$ for the 3D case.
The definition \eqref{ii} is the sum of Hankel functions that converges very slowly, which will cost 
a lot of time in calculation. From the definition of $G(x)$ in \eqref{ei},
the function $G(x)$ is defined by the series of exponential functions and converges rapidly 
when $x$ is away from the line $x_2=0$. When $x_2\rightarrow 0$, the series converges slowly, 
so this method fails. Similarly, the direct calculation from the definition \eqref{he} of $G_d(x)$ 
fails when $x_3\rightarrow 0$.

Many methods have been developed to evaluate the Green's function in 2D and 3D domains, such as
the Kummer's transformation, the Ewald method, the Lattice sums and the integral representations.
Linton compared several methods in 2D domains in \cite{CML1} and in 3D domains in \cite{CML2}. 
In these two papers, Linton gave the conclusions that, if only a few values are needed, 
the Ewald's method is the most efficient method, while if a large number of values are needed, 
the lattice sums method is better. There are many papers concerning the lattice sums method and 
the Ewald's method; see \cite{JM,YY,VGP,CWW,AM,VBB,ASSL,KM,OJW,SCBBM} for the 2D case 
and \cite{MP,TJ,KM,OJW,SCBBM,AM,VBB,AM,LT,SF,CML3} for the 3D case. There are also some other methods 
considering the numerical evaluation of the Green's functions; see, e.g., \cite{BKM,OC,BD,BF}. 
In particular, the cases that are near or at the Wood anomalies, which are very difficult to dealt with, 
have been considered in \cite{OC,BD,BF}. 
For very large $k$'s ($>10^5$), Kurkcu and Reitich introduced the NA method in 2D domains, which comes 
from the integral representations (see \cite{KR}). However, all of the known methods have their advantages 
and limitations. Take the lattice sums, the Ewald method and the NA method in 2D domain for example. 
The lattice sums converge very fast, especially when the wave number $k$ is small enough. 
This makes it the best method to calculate a large number of values in this case. 
But the convergence radius of the lattice sums depends on the wave number $k$. If $k$ becomes larger, 
the lattice sums need more terms to achieve the same accuracy, which involves more evaluations of Hankel 
functions and Bessel functions and thus needs more time.
What is worse, when $k$ is getting larger ($k\geqslant 50$), the lattice sums only converge in a very 
small disk, so the method no longer works for most points. However, the Ewald's method is always able 
to achieve the expected accuracy for both small and large $k$'s, but it costs much more time. 
So it might not be a good choice when a large number of values are needed. The NA method for 2D domains 
is of higher efficiency than the other methods when $k$ is very large ($>10^5$). However, when $k$ is 
not so large, e.g., $k\sim100$, its efficiency is even worse than the Ewald method, so it is not suitable 
for the calculation. The aim of this paper is to seek for a more efficient method that converges for any 
wave number $k$. In this paper, we develop a new FFT-based method that converges for any $k$ and is very 
efficient when a large number of values are needed.

Our method is based on FFT. First, we will transform the Green's function $G(x)$ ($G_d(x)$) into a periodic 
function in 2D (3D) domains. Then we can compute its Fourier series analytically by direct calculations. 
The idea is from \cite{VG1}, in which the Fourier series of the Hankel function are used to solve 
Lippmann-Schwinger equations in 2D domains. The method was extended to 2D periodic domains in \cite{SK} 
and to 3D domains in \cite{TH,LN}.
In one periodic cell, $x=0$ is the only singularity of the Green's function, where the Fourier series do 
not converge. For a better convergence, we will remove the singularity to get the smooth enough modified 
functions. The Fourier coefficients of the modified functions are computed numerically and applied to 
obtain the values on the grid points from IFFT.
Then the values of the modified functions on random points can be evaluated simply by interpolation.
It is now easy to obtain the values of the Green's functions. The derivatives of the Green's functions
can be calculated in the similar way.

This paper is organized as follows. In Section \ref{sec2}, we give a new FFT-based algorithm to
calculate the Green's function $G(x)$ in two dimensions. In Section \ref{sec6.3D}, the FFT-based algorithm
is applied to calculate the Green's function $G_d(x)$ for the Helmholtz equation in
three dimensions. In Section \ref{sec9.3D}, the FFT-based algorithm is extended to calculate the Green's tensor
$\mathbb{G}(x)$ for the Maxwell equations. The convergence analysis and error
estimates are given in Section \ref{sec3} for the numerical methods. 
In Section \ref{sec4}, we present some numerical examples for the
calculation of $G(x),\,G_d(x),\,\mathbb{G}(x)$.

\section{The FFT-based algorithm in two dimensions}\label{sec2}
\setcounter{equation}{0}

\subsection{Periodization and Fourier coefficients}\label{sec2.1}

Set $D_c=[-\pi,\pi]\times[-c,c]$, where $c$ is a positive real number which is chosen so that
for $|x_2|\ge c$ the series expansion (\ref{ei}) can be directly used to compute $G(x)$
efficiently. Thus, we will develop an efficient method to compute the Green's function $G(x)$
for $x\in D_c$.

From the series expansion (\ref{ei}) it follows that
\ben
e^{-i\alpha x_1}G(x)=\frac{i}{4\pi}\sum_{n=-\infty}^\infty\frac{1}{\beta_n}\exp({i n x_1+i\beta_n |x_2|}),
\enn
which is $2\pi$-periodic in the $x_1$-direction.

Let $\wid{c}>c$ and let $\mathcal{X}(t)$ be a smooth cut-off function for \high{$t\in \R$} such that
$\mathcal{X}(t)=1$ for \high{$|t|\le c$} and $\mathcal{X}(t)=0$ for \high{$|t|\ge(c+\wid{c})/2.$}
Define a new function $K(x)$:
\ben
K(x)=e^{-i\alpha x_1}G(x)\mathcal{X}(|x_2|)\quad\mbox{for}\;x\in\R\times[-\wid{c},\wid{c}].
\enn
Then $K(x)$ is zero for $x\in[-\pi,\pi]\times[(c+\wid{c})/2,\wid{c}]$ and
$x\in [-\pi,\pi]\times[-\wid{c},-(c+\wid{c})/2]$.
We extend $K$ into a $2\wid{c}$-periodic function in $x_2$-direction  which is denoted by $K$ again. 
Then $K$ is now a bi-periodic function in $\R^2$.
Note that $K$ is smooth in $D_{\tilde c}\setminus \{0\}$ and
\ben
G(x)=e^{i\alpha x_1}K(x)\quad\text{for}\;\;x\in D_c.
\enn

We now calculate the Fourier coefficients of $K$. To do this, we divide the rectangular
domain $D_{\wid{c}}$ uniformly into $2N\times2N$ small rectangles.
Set $S_N=\{j=(j_1,j_2)\in\Z^2\;:\;-N\le j_1, j_2<N\}$.
The grid points are denoted by $x_j=(x_{1j_1},x_{2j_2})$, where $x_{1j_1}=(\pi/N)j_1$,
$x_{2j_2}=(\wid{c}/N)j_2$ with $j=(j_1,j_2)\in S_N$.
Define the Fourier basis of $L^2(D_{\wid{c}})$ by
\ben
\phi_j(x)=\frac{1}{2\sqrt{\pi\wid{c}}}\exp(ij_1x_1+ij_2\pi x_2/\wid{c}),\quad j=(j_1,j_2)\in\Z^2.
\enn
Then $K(x)$ can be approximated by $K_N(x)$:
\ben\label{fck}
K_N(x)=\sum_{j\in S_N}\hat{K}_j\phi_j(x),
\enn
where $\hat{K}_j$ is the j-th Fourier coefficient of $K$.
By a direct calculation $\hat{K}_j$ is given by
\ben
\hat{K}_j&=&\int_{D_{\wid{c}}}K(x)\ov{\phi}_j(x)dx=\int_{D_{\wid{c}}}K(x)\phi_{-j}(x)dx\\
&=&\frac{i}{4\sqrt{\pi\wid{c}}}\frac{1}{\beta_{j_1}}\left(\int_0^{\wid{c}}
  e^{i(\beta_{j_1}-j_2\pi/\wid{c})x_2}\mathcal{X}(x_2)dx_2
  +\int_0^{\wid{c}}e^{i(\beta_{j_1}+j_2\pi/\wid{c})x_2}\mathcal{X}(x_2)dx_2\right).
\enn
If $\beta_{j_1}\neq\pm j_2\pi/\wid{c}$, then by integration by parts we have
\be\no
\hat{K}_j&=&\frac{1}{2\sqrt{\pi\wid{c}}}\left[\frac{1}{(\alpha+j_1)^2+(j_2\pi/\wid{c})^2-k^2}
 +\frac{1}{2\beta_{j_1}(j_2\pi/\wid{c}-\beta_{j_1})}\int_0^{\wid{c}}
  e^{i\beta_jx_2}\mathcal{X}'(x_2)e^{-i(j_2\pi/\wid{c})x_2}dx_2\right.\\ \label{fck-coef}
&&\left.\qquad-\frac{1}{2\beta_{j_1}(j_2\pi/\wid{c}+\beta_{j_1})}\int_0^{\wid{c}}
  e^{i\beta_jx_2}\mathcal{X}'(x_2)e^{i(j_2\pi/\wid{c})x_2}dx_2\right].
\en
The last two integrals can be calculated by the one-dimensional Fast Fourier Transform (FFT)
and matrix computation.

If $\beta_{j_1}=-j_2\pi/\wid{c}$ with $j_2\not=0$ then we have
\ben
\hat{K}_j=\frac{i}{4\sqrt{\pi\wid{c}}}\frac{1}{\beta_{j_1}}\int_0^{\wid{c}}\mathcal{X}(x_2)dx_2
+\frac{1}{4\sqrt{\pi\wid{c}}\beta_{j_1}(j_2\pi/\wid{c}-\beta_{j_1})}
\left[1+\int_0^{\wid{c}}\mathcal{X}'(x_2)e^{i(\beta_{j_1}-j_2\pi/\wid{c})x_2}dx_2\right].
\enn
The two integrals can be calculated by one-dimensional numerical quadratures.
$\hat{K}_j$ can be calculated similarly for the case $\beta_{j_1}=j_2\pi/\wid{c}$ with $j_2\not=0$.
Note that these two special cases seldom \high{occur}; in fact, they can even be avoided by
a perturbation of $\wid{c}$. Therefore, we assume that these two special cases do not occur.

It should be noted that a direct evaluation of $K$ by using the Fourier coefficients $\hat{K}_j$,
that is, by the approximation $K_N=\sum_{j\in S_N}\hat{K}_j\phi_j$,
will not lead to a convergent result due to the singularity at $x=0$ of $K(x)$.
In the next subsection we will introduce a convergent approximation to $K$ with Fourier coefficients
by removing the singularity of $K$.

\subsection{Removal of the singularity}\label{sec2.2}

We now derive a convergent approximation to $K$ with its Fourier coefficients
by removing the singularity of $K$.

From the image representation \eqref{ii}, $G$ is the sum of a singular function
$(i/4)H_0^{(1)}(k|x|)$ and an analytic function
$(i/4)\sum_{n\neq 0, n\in\Z}e^{i2\pi\alpha n}H_0^{(1)}(kr_n)$.
Then the singularity at zero of $G(x)$ is the same as that of $(i/4)H_0^{(1)}(k|x|)$,
so the singularity at zero of $K(x)$ is equivalent to that of $(i/4)H_0^{(1)}(k|x|)e^{-i\alpha x_1}$.

From the definition of the Hankel function:
\ben
H_0^{(1)}(k|x|)=J_0(k|x|)+iY_0(k|x|),
\enn
where the Bessel functions $J_0$ and $Y_0$ are defined by
\ben
J_0(k|x|)&=&\sum_{m=0}^{+\infty}\frac{(-1)^m}{(m!)^2}\Big(\frac{k|x|}{2}\Big)^{2m},\\
Y_0(k|x|)&=&\frac{2}{\pi}\ln\frac{k|x|}{2}J_0(k|x|)-\frac{1}{\pi}\sum_{m=0}^{\infty}
  \frac{(-1)^m}{(m!)^2}\Big(\frac{k|x|}{2}\Big)^{2m}\cdot 2\psi(m+1),
\enn
with the digamma function $\psi$, we have the asymptotic behavior of $H_0^{(1)}(k|x|)$ at small $|x|$:
\ben
H_0^{(1)}(k|x|)=\frac{2i}{\pi}\ln|x|-\frac{2i}{\pi}\frac{k^2|x|^2}{4}\ln|x|+O(|x|^4\ln|x|).
\enn
Thus, $K(x)$ has the asymptotic behavior at small $|x|$:
\be\no
K(x)&\sim&\frac{i}{4}H_0^{(1)}(k|x|)\cdot e^{-i\alpha x_1}\\ \no
&=&\Big[-\frac{1}{2\pi}\ln|x|+\frac{1}{2\pi}\frac{k^2|x|^2}{4}\ln|x|
   +O(|x|^4\ln|x|)\Big][1-i\alpha x_1+O(x_1^2)]\\ \label{sing}
&=&-\frac{1}{2\pi}\ln|x|+\frac{1}{2\pi}i\alpha x_1\ln|x|+O(|x|^2\ln|x|),\quad |x|\sim 0.
\en
Let $\mathcal{Y}_\vep(t)$ be a smooth function for $t\ge0$ such that $\mathcal{Y}_\vep(t)=1$
for $t\in[0,\vep]$ and $\mathcal{Y}_\vep(t)=0$ for $t\in[2\vep,\infty)$, where $0<\vep\ll1$.
Set
\ben
f_1=-\frac{1}{2\pi}\ln|x|\cdot\mathcal{Y}_\vep(|x|),\quad
f_2=-\frac{1}{2\pi}x_1\ln|x|\cdot\mathcal{Y}_\vep(|x|),\quad x\in D_{\wid{c}}.
\enn
Then $f_1,f_2$ are known functions independent of $k$ and $\alpha$.
\high{Extend the these two function into a $2\pi$-periodic function in $x_1$-direction and $2\wid{c}$-periodic 
function in $x_2$-direction. The functions are denoted by $f_1$ and $f_2$ again. Let $L=K-f_1+i\alpha f_2$.}
Then $L$ is a $C^{\mu}$ function with $0<\mu<2$.

The Fourier coefficients of $L$ are given by
\be\label{fcl-coef}
\hat{L}_j=\hat{K}_j-\hat{F}_{1,j}+i\alpha\hat{F}_{2,j},
\en
where $\hat{F}_{1,j}$ and $\hat{F}_{2,j}$ are the Fourier coefficients of $f_1$ and $f_2$,
respectively. By a direct calculation it can be obtained that
\be\label{fcf1-coef}
\hat{F}_{1,j}&=&\frac{1}{j_1^2+j_2^2\pi^2{\wid{c}}^{-2}}\Big(\frac{1}{2\sqrt{\pi\wid{c}}}
   +\frac{1}{2\pi}\hat{\Phi}_{1,j}\Big),\\ \label{fcf2-coef}
\hat{F}_{2,j}&=&\frac{1}{j_1^2+j_2^2\pi^2{\wid{c}}^{-2}}\Big(-2ij_1\hat{F}_{1,j}
   +\frac{1}{2\pi}\hat{\Phi}_{2,j}\Big),
\en
when $|j|\neq 0$, and
\ben
\hat{F}_{1,0}&=&-\frac{1}{2\sqrt{\pi\wid{c}}}\int_0^{2\vep}t\ln(t)\mathcal{Y}_\vep(t)dt,\\
\hat{F}_{2,0}&=&0.
\enn
Here, $\hat{\Phi}_{1,j}$ and $\hat{\Phi}_{2,j}$ are the Fourier coefficients of
$\Phi_1(x)=(2+\ln|x|)\mathcal{Y}_\vep'(|x|)/|x|+\mathcal{Y}_\vep''(|x|)\ln|x|$ and
$\Phi_2(x)=x_1\Phi_1(x)$, respectively.
Note that $\Phi_1$ and $\Phi_2$ are smooth functions so their Fourier coefficients
can be computed directly by 2D FFT.

Let
\ben\label{fcln}
L_N(x)=\sum_{j\in S_N}\hat{L}_j\phi_j(x),
\enn
then the values of $L_N$ at the grid points can be efficiently computed by the 2D inverse FFT (IFFT).

\begin{remark}\label{remark} {\rm 
Note that the functions $\Phi_1$ and $\Phi_2$ are independent of $k$ and $\alpha$, so $\hat{F}_{1,j}$ 
and $\hat{F}_{2,j}$ can be assumed known for different $k$'s and $\al$'s.
}
\end{remark}

\subsection{Evaluation of the Green's function at arbitrary points}\label{sec2.3}
\setcounter{equation}{0}

We now have the values of $L_N$ at the grid points. For an arbitrary point $x\in D_c$,
we use the 2D interpolation to compute the value of $K(x)$ or $G(x)$.
There are many 2D interpolation methods such as nearest-neighbor interpolation, bilinear interpolation, 
spline interpolation and bicubic interpolation.
In this paper, we will use the bicubic interpolation, which only needs to solve
a system of linear equations of $16$ unknowns. The interpolation error is $O(h^4)$,
where $h=O(N^{-1})$ is the largest grid size.

\begin{algorithm}\label{alg1}
Evaluation of the Green's function $G$. From Remark \ref{remark}, $\hat{F}_{1,j}$ and $\hat{F}_{2,j}$ 
are pre-computed, saved and loaded.
\begin{enumerate}
\item Preparation.
\begin{enumerate}
\item Calculate the Fourier coefficients $\hat{K}_j$ by \eqref{fck-coef}.
\item Calculate the Fourier coefficients $\hat{L}_j$ by \eqref{fcl-coef}
\item Calculate the values of $L_N$ at the grid points by 2D IFFT.
\end{enumerate}
\item Calculation.
\begin{enumerate}
\item Input a point $x$.
\item Check the size of $|x_2|;$\\
if $|x_2|>c$, use the series expansion \eqref{ei};\\
if $|x_2|\le c$, go to (c).
\item Find the unique real number $t\in [-\pi,\pi)$ and the unique integer $n$ such that $x_1=2n\pi+t$.
\item Calculate the value of $L_N$ at the point $(t,x_2)$ by bicubic interpolation with
    the nearest $16$ points. Then calculate the approximate value of $K(t,x_2)$ via
    $K_N(t,x_2)=L_N(t,x_2)+(f_1-i\alpha f_2)(t,x_2)$.
\item Calculate the approximate value of $G(x)$ with $G_N(x)=e^{i\alpha x_1}K_N(t,x_2)$.
\end{enumerate}
\end{enumerate}
\end{algorithm}

It is seen that in Algorithm \ref{alg1}, the preparation step is the most time-consuming one,
but it only needs to run once. This step involves the calculation of the Fourier coefficients
$\hat{K}_j$ and 2D IFFT. The evaluation of the Fourier coefficients $\hat{K}_j$ can be carried
out by 1D FFT and matrix calculation, which is of high efficiency in MATLAB.
The details will be explained in Section \ref{sec4}.

\subsection{Derivatives of the Green function}\label{sec2.4} 

In the integral equation method for solving scattering problems by periodic structures, we also need
to evaluate the first- and second-order derivatives of the quasi-periodic Green's function.
In this subsection, we briefly discuss extension of the above idea to their efficient and
accurate computation.

\subsubsection{First-order derivatives}\label{sec2.4.1}

By the definition of $K$ (see Subsection \ref{sec2.1}) it follows that
$G(x)=\exp({i\alpha x_1})K(x)$ for all $x\in D_c$. Then we have
\ben
\frac{\pa G(x)}{\pa x_1}=i\alpha e^{i\alpha x_1}K(x)+e^{i\alpha x_1}\frac{\pa K(x)}{\pa x_1},\quad
\frac{\pa G(x)}{\pa x_2}=e^{i\alpha x_1}\frac{\pa K(x)}{\pa x_2}.
\enn
Define
\ben
K_1(x)=i\alpha K(x)+\frac{\pa K(x)}{\pa x_1},\quad
K_1(x)=\frac{\pa K(x)}{\pa x_2}.
\enn
Then $K_1(x)=\exp({-i\alpha x_1}){\pa G(x)}/{\pa x_1}$,
$K_2(x)=\exp({-i\alpha x_1}){\pa G(x)}/{\pa x_2}$ for all $x\in D_c$.
The functions $K_1$ and $K_2$ are bi-periodic with the same periods as $K$.
Denote by $\hat{K}_{1,j}$ and $\hat{K}_{2,j}$ the Fourier coefficients of $K_1$ and $K_2$,
respectively. Then $\hat{K}_{1,j}=i(\alpha+j_1)\hat{K}_j$ and
$\hat{K}_{2,j}=ij_2({\pi}/{\wid{c}})\hat{K}_j$.
We now remove the singularity of $K_1$ and $K_2$.
The singularity at zero of $K_1$ is given by
\ben
K_1(x)&\sim& e^{-i\alpha x_1}\frac{i}{4}\frac{\pa H_0^{(1)}(k|x|)}{\pa x_1}
   =-\frac{ik}{4}e^{-i\alpha x_1}H_1^{{1}}(k|x|)\frac{x_1}{|x|}\\
&=&-\frac{ik}{4}\Big(1-i\alpha x_1-\frac{\alpha^2}{2}x_1^2+O(x_1^3)\Big)\Big(\frac{ik}{\pi}x_1\ln|x|
   -\frac{2i}{\pi}\frac{x_1}{k|x|^2}+O(|x|^2\ln|x|)\Big)\\
&=&-\frac{1}{2\pi}\Big(-\frac{k^2}{2}x_1\ln|x|+\frac{x_1}{|x|^2}-i\alpha\frac{x_1^2}{|x|^2}
  -\frac{\alpha^2}{2}\frac{x_1^3}{|x|^2}\Big)+O(|x|^2\ln|x|).
\enn
Define the singular function
\ben
h_1(x)=-\frac{1}{2\pi}\Big(-\frac{k^2}{2}x_1\ln|x|+\frac{x_1}{|x|^2}-i\alpha\frac{x_1^2}{|x|^2}
-\frac{\alpha^2}{2}\frac{x_1^3}{|x|^2}\Big)\mathcal{Y}_\vep(|x|),
\enn
where $\mathcal{Y}_\vep$ is defined in Subsection \ref{sec2.2}.
Then $K_1(x)=[K_1(x)-h_1(x)]+h_1(x)$ where the function $K_1(x)-h_1(x)$ has a better regularity.
The Fourier coefficients $\hat{H}_{1,j}$ of $h_1$ can be calculated directly,
similarly as for $\hat{F}_{1,j}$ and $\hat{F}_{2,j}$ in Subsection \ref{sec2.2}.
Then the function $K_1-h_1$ can be accurately evaluated efficiently at the grid points
via the 2D IFFT by using the known Fourier coefficients $\hat{K}_{1,j}-\hat{H}_{1,j}$.
Finally, the function $K_1$ or ${\pa G(x)}/{\pa x_1}$ can be accurately evaluated efficiently
at an arbitrary point $x\in D_c$ by bi-cubic interpolation.

Similarly, the function $K_2$ or ${\pa G(x)}/{\pa x_2}$ can also be accurately evaluated efficiently.
In this case, the function $K_2(x)=[K_2(x)-h_2(x)]+h_2(x)$ with a regular function $K_2(x)-h_2(x)$,
where the singular function
\ben
h_2(x)=-\frac{1}{2\pi}\Big(-\frac{k^2}{2}x_1\ln|x|+\frac{x_2}{|x|^2}-i\alpha\frac{x_1x_2}{|x|^2}
-\frac{\alpha^2}{2}\frac{x_1^2x_2}{|x|^2}\Big)\mathcal{Y}_\vep(|x|).
\enn

\subsubsection{Second-order derivatives}\label{sec2.4.2}

The second-order derivatives of $G$ usually occurs in the hyper-singular integral operator $T$
defined by
\ben
(T\phi)(x)=\frac{\pa}{\pa\nu(x)}\int_{\G}\frac{\pa G}{\pa\nu(y)}(x,y)\phi(y)ds(y)
\enn
for $\phi$ in certain space, where $\nu(x)$ is the unit normal vector at $x\in\G$.
In the integral equation methods for wave scattering from periodic structures,
we usually need to compute the difference between two $T$ operators with different
wave numbers. Thus, we will discuss the accurate and efficient evaluation of the differences
of second-order derivatives of $G_{k_1}$ and $G_{k_2}$ for different wave numbers $k_1$ and $k_2$
rather than the second-order derivatives of $G$ itself, where $G_{k_j}$ is defined similarly
as $G$ with $k$ replaced by $k_j$, $j=1,2.$

The second-order derivatives of $G$ in $D_c$ can be calculated as follows:
\ben
\frac{\pa^2 G(x)}{\pa x_1^2}&=&-\alpha^2 e^{i\alpha x_1}K(x)+2i\alpha e^{i\alpha x_1}
\frac{\pa K(x)}{\pa x_1}+e^{i\alpha x_1}\frac{\pa^2 K(x)}{\pa x_1^2};\\
\frac{\pa^2 G(x)}{\pa x_1\pa x_2}&=&i\alpha e^{i\alpha x_1}\frac{\pa K}{\pa x_2}
+e^{i\alpha x_1}\frac{\pa^2 K(x)}{\pa x_1\pa x_2};\\
\frac{\pa^2 G(x)}{\pa x_1\pa x_2}&=&e^{i\alpha x_1}\frac{\pa^2 K(x)}{\pa x_2^2}.
\enn
Define 
\ben
K_{11}(x)&=&-\alpha^2 K(x)+2i\alpha\frac{\pa K(x)}{\pa x_1}+\frac{\pa^2 K(x)}{\pa x_1^2};\\
K_{12}(x)&=&i\alpha\frac{\pa K}{\pa x_2}+\frac{\pa^2 K(x)}{\pa x_1\pa x_2};\\
K_{22}(x)&=&\frac{\pa^2 K(x)}{\pa x_2^2}.
\enn
Then $K_{11}(x)=\exp({-i\alpha x_1}){\pa^2 G(x)}/{\pa x_1^2}$,
$K_{12}(x)=\exp({-i\al x_1}){\pa^2 G(x)}/{\pa x_1\pa x_2}$ and
$K_{22}(x)=\exp({-i\al x_1}){\pa^2 G(x)}/{\pa x_2^2}$ in $D_c$, which are bi-periodic.
Their Fourier coefficients are given as $\hat{K}_{11,j}=-(\alpha+j_1)^2\hat{K}_j$,
$\hat{K}_{12,j}=-(\alpha+j_1) j_2({\pi}/{\wid{c}})\hat{K}_j$ and
$\hat{K}_{22,j}=-(j_2{\pi}/{\wid{c}})^2\hat{K}_j$.
When $|x|\to0$, we have
\ben
K_{11}(x)&\sim& e^{-i\alpha x_1}\frac{\pa^2 G(x)}{\pa x_1^2}
          \sim\frac{i}{4}e^{-i\alpha x_1}\frac{\pa^2 H_0^{(1)}(k|x|)}{\pa x_1^2}\\
&=&\frac{ik}{4}e^{-i\alpha x_1}\Big[H_1^{(1)}(k|x|)\frac{|x|^2-x_1^2}{|x|^3}
  +\frac{kx_1^2}{2|x|^2}\big[H_0^{(1)}(k|x|)-H_2^{(1)}(k|x|)\big]\Big]\\
&=&\frac{k^2}{2\pi}e^{-i\alpha x_1}\Big(\frac{1}{2}\ln|x|+\frac{x_1^2}{2|x|^2}\Big)
  +\frac{e^{-i\alpha x_1}}{\pi}\Big(\frac{1}{|x|^2}-\frac{2x_1^2}{|x|^4}\Big)+O(|x|^2\ln|x|).
\enn
Let $T_{11}(x)=K_{k_1,11}(x)-K_{k_2,11}(x)$, where $K_{k_j,11}$ is defined similarly as $K_{11}$
with $k$ replaced by $k_j$, $j=1,2.$ Then the asymptotic behavior of $T_{11}$ at $x=0$ is
\ben
T_{11}(x)&\sim&\frac{k_1^2-k_2^2}{2\pi}e^{-i\alpha x_1}\Big(\frac{1}{2}\ln|x|+\frac{x_1^2}{2|x|^2}
  +O(|x|^2\ln|x|)\Big)\\
&=&-\frac{k_1^2-k_2^2}{2\pi}\Big(\frac{1}{2}\ln|x|-\frac{i}{2}\alpha x_1\ln|x|+\frac{x_1^2}{2|x|^2}
  -\frac{i\alpha x_1^3}{2|x|^2}\Big)+O(|x|^2\ln|x|).
\enn
Define
\ben
h_{11}(x)=-\frac{k_1^2-k_2^2}{2\pi}\Big(\frac{1}{2}\ln|x|-\frac{i}{2}\alpha x_1\ln|x|
+\frac{x_1^2}{2|x|^2}-\frac{i\alpha x_1^3}{2|x|^2}\Big)\mathcal{Y}_\vep(|x|).
\enn
Then $T_{11}=(T_{11}-h_{11})+h_{11}$ with a regular function $T_{11}-h_{11}$.
The Fourier coefficients $\hat{H}_{11}$ of $h_{11}$ can be calculated directly as
in Subsection \ref{sec2.2} (cf. (\ref{fcf1-coef}) and (\ref{fcf2-coef})),
and the Fourier coefficients of $T_{11}-h_{11}$ are given as
$\hat{T}_{11}-\hat{H}_{11}=\hat{K}^{(k_1)}_{11,j}-\hat{K}^{(k_2)}_{11,j}-\hat{H}_{11}$,
where $\hat{K}^{(k_i)}_{11,j}$ is the same as $\hat{K}_{11,j}$ with $k$ replaced by $k_i$ ($i=1,2$).

The regular function $T_{11}-h_{11}$ can be accurately evaluated efficiently at the grid points
via the 2D IFFT, by using the known Fourier coefficients $\hat{T}_{11}-\hat{H}_{11}$.
Then at an arbitrary point $x\in D_c$ the function $T_{11}(x)-h_{11}(x)$ can be evaluated efficiently 
and accurately by bi-cubic interpolation. Thus, the function $T_{11}(x)$ can be computed efficiently
and accurately at an arbitrary point $x\in D_c$ by bi-cubic interpolation
since $T_{11}(x)=[T_{11}(x)-h_{11}(x)]+h_{11}(x)$.
Similarly as before, it can be shown that this computation method is of second-order convergence
with the error bound $C{|k_1^4-k_2^4|}/{N^2}$ since the leading term of the Fourier coefficients
$\hat{T}_{11}-\hat{H}_{11}$ is
\ben
\hat{T}_{11}-\hat{H}_{11}\sim\frac{k_1^4-k_2^4}
{[(\al+j_1)^2+(j_2\pi/{\wid{c}})^2-k_1^2][(\al+j_1)^2+(j_2{\pi}/{\wid{c}})^2-k_2^2]}
\enn
for large $j_1,j_2$.

The functions $T_{12}=K_{k_1,12}-K_{k_2,12}$ and $T_{22}=K_{k_1,22}-K_{k_2,22}$ can be evaluated
similarly since $T_{12}=(T_{12}-h_{12})+h_{12}$ and $T_{22}=(T_{22}-h_{22})+h_{22}$ with regular
functions $T_{12}-h_{12}$ and $T_{22}-h_{22}$ and singular functions
\ben
h_{12}(x)&=&-\frac{k_1^2-k_2^2}{2\pi}\Big(\frac{x_1x_2}{2|x|^2}
-\frac{i\alpha}{2}\frac{x_1^2x_2}{|x|^2}\Big)\mathcal{Y}_\vep(|x|);\\
h_{22}(x)&=&-\frac{k_1^2-k_2^2}{2\pi}\Big(\frac{1}{2}\ln|x|-\frac{1}{2}i\alpha x_1\ln|x|
+\frac{x_2^2}{2|x|^2}-\frac{i\alpha}{2}\frac{x_1x_2^2}{|x|^2}\Big)\mathcal{Y}_\vep(|x|).
\enn
Here, $K_{k_j,12}$ and $K_{k_j,22}$ are defined similarly as $K_{12}$ and $K_{22}$, respectively,
with $k$ replaced by $k_j$, $j=1,2.$

\section{The FFT-based algorithm in three dimensions}\label{sec6.3D}
\setcounter{equation}{0}

\subsection{Periodization and Fourier coefficients}\label{sec6.3D.1}

Set $D_{dc}=[-\pi,\pi]\times[-\pi,\pi]\times[-c,c]$, where $c$ is a chosen positive constant
such that for $|x_3|\ge c$ the series representation \eqref{he} can be directly used to evaluate
$G_d$ efficiently and accurately. Thus, we only consider the case when $x\in D_{dc}$.

By \eqref{he} we know that
\ben
e^{-i\alpha_1x_1-i\alpha_2x_2}G_d(x)=\frac{i}{8\pi^2}\sum_{n_1,n_2\in\Z}
\frac{1}{\beta_{n_1,n_2}}e^{in_1x_1+in_2x_2+i\beta_{n_1,n_2}|x_3|},
\enn
which is periodic in both $x_1$ and $x_2$ with period $2\pi.$
Now choose $\wid{c}>c$ and set $D_{d\wid{c}}=[-\pi,\pi]\times[-\pi,\pi]\times[-\wid{c},\wid{c}]$.
Define
\ben
K_d(x)=e^{-i\alpha_1x_1-i\alpha_2x_2}G_d(x)\mathcal{X}(|x_3|)\quad\mbox{for}\;\;x\in D_{d\wid{c}},
\enn
and extend it into a periodic function in $\R^3$, denoted by $K_d$ again.
Note that $K_d$ is smooth in $D_{d\wid{c}}\backslash\{0\}$ and
\be\label{pr}
G_d(x)=e^{i\alpha_1x_1+i\alpha_2x_2}K_d(x)\quad\mbox{for}\;\;x\in D_{dc}.
\en

We now calculate the Fourier coefficients of $K_d(x)$. For a large integer $N>0$ divide
$D_{d\wid{c}}$ into $2N\times 2N\times 2N$ uniform grids.
Set $S_N=\{l=(l_1,l_2,l_3)\in\Z^3:\,-N\le l_1,l_2,l_3< N\}$.
For $j=(j_1,j_2,j_3)\in S_N$ denote by $x_j=(x_{1,j_1},x_{2,j_2},x_{3,j_3})$ the grid points,
where $x_{1,j_1}=j_1{\pi}/{N}$, $x_{2,j_2}=j_2{\pi}/{N}$, $x_{3,j_3}=j_3{\wid{c}}/{N}$.
Let
\ben
\phi_j(x)=\frac{1}{\sqrt{8\pi^2\wid{c}}}\exp[i(j_1x_2+j_2x_2+j_3x_3/\wid{c})],
\quad j=(j_1,j_2,j_3)\in\Z^3.
\enn
Then $\{\phi_{j}\}_{j\in\Z^3}$ is the Fourier basis of $L^2(D_{d\wid{c}})$ and
the Fourier series of $K_d$ is given as $K_d(x)=\sum_{j\in\Z^3}\hat{K}_{dj}\phi_j(x)$,
where $\hat{K}_{dj}$ is the $j$-th Fourier coefficient of $K_d$ given by
\be\no
\hat{K}_{dj}&=&\int_{D_{d\wid{c}}}K_d(x)\phi_{-j}(x)dx
=\int_{D_{d\wid{c}}}e^{-i\alpha_1x_1-i\alpha_2x_2}G_d(x)\mathcal{X}(|x_3|)\phi_{-j}(x)dx\\ \no
&=&\frac{1}{\sqrt{8\pi^2\wid{c}}}\frac{i}{8\pi^2}\int_{D_{d\wid{c}}}
  \sum_{n_1,n_2\in\Z}\frac{1}{\beta_{n_1,n_2}}e^{in_1x_1+in_2x_2+i\beta_{n_1,n_2}|x_3|}
  \mathcal{X}(|x_3|)e^{-ij_1x_2-ij_2x_2-i{j_3}x_3/{\wid{c}}}dx\\ \label{kdj}
&=&\frac{i}{2\sqrt{8\pi^2\wid{c}}}\frac{1}{\beta_{j_1,j_2}}
\int_0^{\wid{c}}\left[e^{i(\beta_{j_1,j_2}-j_3/\wid{c})x_3}
    +e^{-i(\beta_{j_1,j_2}-j_3/\wid{c})x_3}\right]\mathcal{X}(x_3)dx_3.
\en

If $\beta_{j_1,j_2}\neq\pm j_3{\pi}/{\wid{c}}$ it then follows by integration by parts that
\be\no
\hat{K}_{dj}&=&\frac{1}{\sqrt{8\pi^2\wid{c}}}\Big[\frac{1}{(\alpha_1+j_1)^2
  +(\alpha_2+j_2)^2+(j_3\pi/\wid{c})^2-k^2}\\
&&\qquad\qquad+\frac{1}{2\beta_{j_1,j_2}(j_3{\pi}/{\wid{c}}-\beta_{j_1,j_2})}
  \int_{0}^{\wid{c}}e^{i(\beta_{j_1,j_2}-j_3\pi/\wid{c})x_3}\mathcal{X}'(x_3)dx_3\\ \label{kdj1}
&&\qquad\qquad-\frac{1}{2\beta_{j_1,j_2}(j_3{\pi}/{\wid{c}}+\beta_{j_1,j_2})}
  \int_{0}^{\wid{c}}e^{i(\beta_{j_1,j_2}+j_3\pi/\wid{c})x_3}\mathcal{X}'(x_3)dx_3\Big].
\en
The two integrals can be calculated efficiently by 1D FFT.

If $\beta_{j_1,j_2}=-j_3\pi/\wid{c}$ we have
\be\no
\hat{K}_{dj}&=&\frac{i}{2\sqrt{8\pi^2\wid{c}}\beta_{j_1,j_2}}\int_0^{\wid{c}}\mathcal{X}(x_3)dx_3\\
&&+\frac{1}{2\sqrt{8\pi^2\wid{c}}\beta_{j_1,j_2}(j_3{\pi}/{\wid{c}}-\beta_{j_1,j_2})}
\Big(1+\int_{0}^{\wid{c}}\mathcal{X}'(x_3)e^{i(\beta_{j_1,j_2}-ij_3{\pi}/{\wid{c}})x_3}dx_3\Big).
\label{kdj2}
\en
The two integrals can be calculated by 1D FFT or 1D numerical quadratures.
The case with $\beta_{j_1,j_2}=j_3{\pi}/{\wid{c}}$ can be dealt with similarly.
Note that these cases can be avoided by choosing $\wid{c}$ properly, so we assume that these cases 
do not occur in this paper.

Similar to the 2D case, the direct evaluation of $K_d$ by using the Fourier coefficients 
$K_{dN}(x)=\sum_{j\in S_N}\hat{K}_{dj}\phi_j(x)$ for $x\in D_{dc}$, will not lead to a convergent result
due to the singularity at $x=0$ of $K_d$. In the next subsection we will introduce a convergent 
approximation to $K_d$ by removing the singularity of $K_d$.

\subsection{Removal of the singularity}

Since $G_d(x)$ can be written as $G_d(x)=\exp(ik|x|)(4\pi|x|)^{-1}+a(x)$
in the neighborhood of $x=0$, where $a$ is an analytic function,
then the singularity of $K_d(x)$ at $x=0$ is the same as
$\exp[-i(\al_1x_1+\al_2x_2)]\exp(ik|x|)(4\pi|x|)^{-1}$.
Define
\ben
F(x)=e^{-i\alpha_1x_1-i\alpha_2x_2}\frac{e^{ik|x|}}{4\pi|x|}\mathcal{Y}_\vep(|x|),
\quad x\in D_{d\wid{c}}
\enn
for some positive number $\vep\ll\min\{1,\wid{c}\}/2$.
Then $F$ is zero outside of $D_{d\wid{c}}$ and can be extended into a periodic function
in $R^3$, denoted by $F$ again, with the same period as $K_d$. Thus, $L_d(x):=K_d(x)-F(x)$
is smooth in $D_{d\wid{c}}$ and periodic in $\R^3$.
The $j-$th Fourier coefficient of $L_d$ is given as $\hat{L}_{dj}=\hat{K}_{dj}-\hat{F}_j$,
where $\hat{K}_{dj}$ and $\hat{F}_j$ are the $j-$th Fourier coefficients of $K_d$ and $F$,
respectively. By a similar argument as in \cite{VG1} it can be derived that
\be\no
\hat{F}_j&=&\frac{1}{\sqrt{8\pi^2\wid{c}}}\frac{1}{(\al_1+j_1)^2+(\al_2+j_2)^2
       +(j_3\pi/\wid{c})^2-k^2}\\ \label{fj}
&&\cdot\Big[1-\int_{D_{d\wid{c}}}\frac{e^{ik|x|}}{4\pi|x|}\left[2ik\mathcal{Y}'_\vep(|x|)
+\mathcal{Y}''_\vep(|x|)\right]e^{-i(\al_{1,j_1}x_1+\al_{2,j_2}x_2+j_3 x_3\pi/\wid{c})}dx\Big].
\en
The integral can be calculated efficiently with 3D FFT.

Define the truncated Fourier series
\be\label{ln}
L_{dN}(x)=\sum_{j\in S_N}\hat{L}_{dj}\phi_j(x).
\en
Then $L_{dN}$ converges to $L$ with the convergence rate $O(N^{-m})$ for any $m\in\N^+$,
as $N\to\infty$ (see Theorem \ref{h1} in Section \ref{sec3}).
The values of $L_{dN}$ on the grid points can be evaluated efficiently by 3D IFFT.

\subsection{Evaluation of the Green function at arbitrary points}\label{sec6.3D.3}

The value of $L_{dN}(x)$ at an arbitrary point $x$ can be computed from the values of $L_{dN}$
at the grid points by using tri-cubic interpolation (e.g., the algorithm discussed in \cite{LM}).
The relative error of the tri-cubic interpolation is $h^4$, where $h=O(1/N)$ is the largest grid size.
Thus, for an arbitrary point $x\in D_{dc}$ the function $K_d(x)$ or the Green function $G_d(x)$
can be evaluated efficiently and accurately since $K_d(x)=L_d(x)+F(x)$.
This is summarized in the following algorithm.

\begin{algorithm}\label{alg2.3D}
The evaluation of the Green function $G_d(x)$.
\begin{enumerate}
\item Preparation
\begin{enumerate}
\item Calculate the coefficients $\hat{K}_{dj}$ by (\ref{kdj})-(\ref{kdj2}).
\item Calculate the coefficients $\hat{L}_{dj}$ via the formula
      $\hat{L}_{dj}=\hat{K}_{dj}-\hat{F}_j$ and (\ref{fj}).
\item Calculate the values of $L_{dN}$ on the grid points by 3D IFFT.
\end{enumerate}
\item Calculation
\begin{enumerate}
\item Input a point $x$.
\item Check the size of $|x_3|$;\\
if $|x_3|>c$, use \eqref{he} to calculate the value $G_d(x)$;\\
if $|x_3|\le c$, go to (c).
\item Find the unique real numbers $t_1,t_2\in[-\pi,\pi)$ and integers $n_1,n_2$
such that $x_1=2n_1\pi+t_1$, $x_2=2n_2\pi+t_2$.
\item Calculate the value $L_{dN}(t_1,t_2,x_3)$ by tri-cubic interpolation.
\item Calculate the approximate value of $K_d(t_1,t_2,x_3)$ via
      $K_{dN}(t_1,t_2,x_3)=L_{dN}(t_1,t_2,x_3)+F(t_1,t_2,x_3)$.
\item Calculate the approximate value of $G_d(t_1,t_2,x_3)$ by the formula
$G_{dN}(t_1,t_2,x_3)=e^{i\alpha_1t_1+i\alpha_2t_2}K_{dN}(t_1,t_2,x_3)$.
\end{enumerate}
\end{enumerate}
\end{algorithm}

In this algorithm, the preparation step is the most time-consuming one; however,
it needs to run only once.
The calculation step mainly needs to compute a multiplication of a $64\times 64$ matrix with
a $64$-dimensional vector and therefore consumes very little time.
The details will be discussed in Section \ref{sec4}.

\section{The FFT-based algorithm in the Maxwell case}\label{sec9.3D}
\setcounter{equation}{0}

The doubly quasi-periodic Green's tensor $\mathbb{G}(x)$ for Maxwell's equations in 3D can be rewritten in the form
\be\label{mg}
\mathbb{G}(x)=G_d(x)\mathbb{I}+\frac{1}{k^2}\left(
\begin{array}{ccc}
\ds\frac{\pa^2 G_d}{\pa x_1^2} & \ds\frac{\pa^2 G_d}{\pa x_1\pa x_2}
   &\ds\frac{\pa^2 G_d}{\pa x_1\pa x_3}\\
\ds\frac{\pa^2 G_d}{\pa x_2\pa x_1} & \ds\frac{\pa^2 G_d}{\pa x_2^2}
   & \ds\frac{\pa^2 G_d}{\pa x_2\pa x_3}\\
\ds\frac{\pa^2 G_d}{\pa x_3\pa x_1} & \ds\frac{\pa^2 G_d}{\pa x_3\pa x_2}
   & \ds\frac{\pa^2 G_d}{\pa x_3^2}\\
\end{array}\right).
\en
In addition to the values of the 3D Green's function $G_d(x)$, we only need to evaluate 
the second-order derivatives ${\pa^2 G_d}/({\pa x_p\pa x_q})$, $p,q=1,2,3$. Define
\ben
K^{pq}(x)=e^{-i\al_1x_1-i\al_2x_2}\frac{\pa^2}{\pa x_p\pa x_q}\left[G_d(x)\mathcal{X}(|x_3|)\right].
\enn
Then
\be\label{pre}
\frac{\pa^2 G_d(x)}{\pa x_p\pa x_q}=e^{i\al_1x_1+i\al_2x_2}K^{pq}(x),\quad x\in D_c.
\en
From the definition of $K_d$ it follows that
\ben
K^{pq}&=&e^{-i\al_1x_1-i\al_2x_2}\frac{\pa^2}{\pa x_p\pa x_q}
        \left[e^{i\al_1x_1+i\al_2x_2}K_d(x)\right]\\
&=&e^{-i\al_1x_1-i\al_2x_2}\frac{\pa^2}{\pa x_p\pa x_q}
   \left[e^{i\al_1x_1+i\al_2x_2}\sum_{j\in\Z^3}\hat{K}_{dj}\phi_j(x)\right]\\
&=&\frac{1}{\sqrt{8\pi^2\wid{c}}}e^{-i\alpha_1x_1-i\alpha_2x_2}
 \frac{\pa^2}{\pa x_p\pa x_q}\sum_{j\in\Z^3}\hat{K}_{dj}
e^{i\al_{1,j_1}x_1+i\al_{2,j_2}x_2+ij_3x_3\pi/\wid{c}}.
\enn
Set $c^1_j=\al_{1,j_1},\;c^2_j=\al_{2,j_2},\;c^3_j=j_3\pi/\wid{c}$.
Then, by a direct calculation the Fourier coefficients of $K^{pq}$ are given as 
$\hat{K}^{pq}_j=-c^p_jc^q_j\hat{K}_{dj}.$

Similar to the Helmholtz equation case, we also need to remove the singularity of $K^{pq}$.
For $x\in D_{d\wid{c}}$, let
\ben
F^{pq}(x)&=&e^{-i\al_1x_1-i\al_2x_2}\frac{\pa^2}{\pa x_p\pa x_q}
\left[e^{i\al_1x_1+i\al_2x_2}F(x)\right],\\
L_d^{pq}(x)&=&e^{-i\al_1x_1-i\al_2x_2}\frac{\pa^2}{\pa x_p\pa x_q}
\left[e^{i\al_1x_1+i\al_2x_2}L_d(x)\right],
\enn
where $F$ and $L_d$ are defined in Section \ref{sec6.3D} and satisfy that
$L_d(x)=K_d(x)-F(x)$ is smooth in $D_{d\wid{c}}$ and periodic in $\R^3$.
Then $L_d^{pq}(x)=K_d^{pq}(x)-F^{pq}(x)$ is also smooth in $D_{d\wid{c}}$ and periodic in $\R^3$,
and the Fourier coefficients of $F^{pq}$ and $L^{pq}$ are given, respectively, as
$\hat{F}^{pq}_j=-c^p_jc^q_j\hat{F}_j$ and $\hat{L}^{pq}_{dj}=-c^p_jc^q_j\hat{L}_{dj}$.

With these Fourier coefficients, the values of the truncated Fourier series
$L^{pq}_{dN}(x)=\sum_{j\in S_N}\hat{L}^{pq}_{dj}\phi_j(x)$ at the grid points
can be efficiently evaluated by 3D IFFT.
From these values of $L^{pq}_{dN}(x)$ at the grid points, and by tricubic interpolation
the value $L^{pq}_{dN}(x)$ at an arbitrary point $x\in D_{dc}$ can be efficiently computed.
Finally, for an arbitrary point $x\in D_{dc}$ the approximate value of $K^{pq}(x)$
(via the formula $K^{pq}_{dN}(x):=L^{pq}_{dN}(x)+F^{pq}(x)$) and therefore
${\pa^2 G_d(x)}/({\pa x_p\pa x_q})$ (via \eqref{pre}) can be efficiently obtained.

Based on the above idea and Algorithm \ref{alg2.3D}, the following algorithm is given
to evaluate the Green function $\mathbb{G}(x)$ efficiently and accurately.

\begin{algorithm}\label{alg3.3D}
Evaluation of the Green function $\mathbb{G}(x)$.
\begin{enumerate}
\item Preparation
\begin{enumerate}
\item Calculate the coefficients $\hat{K}_{dj}$ by (\ref{kdj})-(\ref{kdj2}).
\item Calculate the coefficients $\hat{L}_{dj}$ via the formula
      $\hat{L}_{dj}=\hat{K}_{dj}-\hat{F}_j$ and (\ref{fj}).
\item Calculate the coefficients $\hat{L}^{pq}_{dj}$ via the formula
      $\hat{L}^{pq}_{dj}=-c^p_jc^q_j\hat{L}_{dj}$.
\item Calculate the values of $L_{dN}$ and $L^{pq}_{dN}$ at the grid points by 3D IFFT.
\end{enumerate}
\item Calculation
\begin{enumerate}
\item Input a point $x$.
\item Check the size of $|x_3|$;\\
if $|x_3|>c$, use \eqref{he} and its derivatives to calculate the value $\mathbb{G}(x)$;\\
if $|x_3|\le c$, go to (c).
\item Find the unique real numbers $t_1,t_2\in[-\pi,\pi)$ and integers $n_1,n_2$
      such that $x_1=2n_1\pi+t_1$, $x_2=2n_2\pi+t_2$.
\item Calculate the value $L_{dN}(t_1,t_2,x_3)$ and $L^{pq}_N(t_1,t_2,x_3)$
      by tri-cubic interpolation.
\item Calculate the approximate value of $K_d(t_1,t_2,x_3)$ via
      $K_{dN}(t_1,t_2,x_3)=L_{dN}(t_1,t_2,x_3)+F(t_1,t_2,x_3)$.
\item Calculate the approximate value of $K^{pq}(t_1,t_2,x_3)$ via
      $K^{pq}_N(t_1,t_2,x_3)=L^{pq}_{dN}(t_1,t_2,x_3)+F^{pq}(t_1,t_2,x_3)$.
\item Calculate the approximate value of $G_d(t_1,t_2,x_3)$ by the formula
$G_{dN}(t_1,t_2,x_3)=e^{i\alpha_1t_1+i\alpha_2t_2}K_d(t_1,t_2,x_3)$.
\item Calculate the approximate value of $\pa^2G_d(t_1,t_2,x_3)/({\pa x_p\pa x_q})$
via the formula $\pa^2G_{dN}(t_1,t_2,x_3)/({\pa x_p\pa x_q})
 =e^{i\al_1t_1+i\al_2t_2}K_{dN}(t_1,t_2,x_3)$.
\item Calculate the approximate value of $\mathbb{G}(x)$ by \eqref{mg}.
\end{enumerate}
\end{enumerate}
\end{algorithm}

Similar to Algorithm \ref{alg2.3D}, in this algorithm, the most time-consuming step
is also the preparation step, which only needs to run once.
The calculation step consists mainly of seven interpolation processes and is very efficient.

\section{Convergence analysis}\label{sec3}
\setcounter{equation}{0}

In this section, we will study the convergence and error estimates of Algorithms \ref{alg1}, \ref{alg2.3D}
and \ref{alg3.3D}.

\subsection{The two-dimensional case}\label{con-2D}

\begin{theorem}\label{thm1}
For any wave number $k>0$ and any positive integer $N$, $L_N$ converges uniformly to $L$
with second-order accuracy as $N\to\infty$, that is,
\ben
\|L_N-L\|_{L^{\infty}(D_{\wid{c}})}=O\left(\frac{k^2+1}{N^2}\right).
\enn
\end{theorem}

\begin{proof}
From \eqref{fck-coef}, \eqref{fcf1-coef} and \eqref{fcf2-coef}, the leading order of
$\hat{K}_j,$ $\hat{F}_{1,j}$ and $\hat{F}_{2,j}$ for $|j|\rightarrow\infty$ is:
\ben
\hat{K}_j&\sim&\frac{1}{2\sqrt{\pi\wid{c}}}\cdot\frac{1}{(\alpha+j_1)^2
    +(j_2\pi/\wid{c})^2-k^2};\\
\hat{F}_{1,j}&\sim&\frac{1}{2\sqrt{\pi\wid{c}}}\cdot\frac{1}{j_1^2+(j_2\pi/\wid{c})^2};\\
\hat{F}_{2,j}&\sim&-\frac{1}{\sqrt{\pi\wid{c}}}
\cdot\frac{ij_1}{\big(j_1^2+(j_2\pi/\wid{c})^2\big)^2}.
\enn
Then the leading order of $\hat{L}_j$ is
\ben
\hat{L}_j=\hat{K}_j-\hat{F}_{1,j}+i\alpha \hat{F}_{2,j}\sim\frac{(k^2-\alpha^2)
\big(j_1^2+(j_2\pi/\wid{c})^2\big)+4\alpha^2j_1^2+2\alpha^3j_1-2\alpha j_1k^2}
{2\sqrt{\pi\wid{c}}\big((\alpha+j_1)^2+(j_2\pi/\wid{c})^2-k^2\big)\big(j_1^2
+(j_2\pi/\wid{c})^2\big)^2}.
\enn
Thus, for any $x\in D_{\wid{c}}$ we have
\ben
|L(x)-L_N(x)|&=&\Big|\sum_{j\notin S_N}\hat{L}_j\phi_j(x)\Big|
\le\frac{1}{2\sqrt{\pi\wid{c}}}\sum_{j\notin S_N}|\hat{L}_j|\\
&\sim&\frac{1}{4\pi\wid{c}}\sum_{j\notin S_N}\frac{|(k^2-\alpha^2)
\big(j_1^2+(j_2\pi/\wid{c})^2\big)+4\alpha^2j_1^2
+2\alpha^3j_1-2\alpha j_1k^2|}{\big((\alpha+j_1)^2
+(j_2\pi/\wid{c})^2-k^2\big)\big(j_1^2+(j_2\pi/\wid{c})^2\big)^2}\\
&=&O\Big(\frac{k^2+1}{N^2}\Big).
\enn
The proof is completed.
\end{proof}

\begin{theorem}\label{thm2}
For any $x\in D_{c}$ the error between the exact value $G(x)$ and the numerical value
$\wid{G}_N(x)$ obtained by Algorithm \ref{alg1} satisfies the following estimate
\ben
|G(x)-\wid{G}_N(x)|\le C\Big(\frac{k^2+2}{N^2}\Big),
\enn
where $C$ is a constant independent of $k$ and $N$.
\end{theorem}

\begin{proof}
By Theorem \ref{thm1} we have that at any grid point $x^*$
\ben
|L(x^*)-L_N(x^*)|\le C\left(\frac{k^2+1}{N^2}\right).
\enn
Note that the procedure of computing $L_N(x^*)$ makes use of FFT to compute the
Fourier coefficients of $L$ with the error bounded above by $CN^{-2}$.
So the value we actually obtained is $\wid{L}_N(x^*)$ with the Fourier coefficients
$\hat{\wid{L}}_j$ satisfying that ${|\hat{L}_j-\hat{\wid{L}}_j|}/{|\hat{\wid{L}}_j|}\le CN^{-2}.$
Thus we have
\ben
|L(x^*)-\wid{L}_N(x^*)|&\le&|L(x^*)-L_N(x^*)|+|L_N(x^*)-\wid{L}_N(x^*)| \\
&\le& C\left(\frac{k^2+1}{N^2}\right)+C\left(\frac{1}{N^{2}}\right)
=C\left(\frac{k^2+2}{N^2}\right).
\enn

For any $x\in D_{c}$ the value $\wid{L}_N(x)$ is obtained by bi-cubic interpolation
with the interpolation error $O(N^{-4})$. Then
\ben
|L(x)-\wid{L}_N(x)|\le C\frac{k^2+2}{N^2}(1+O(N^{-4}))=C\frac{k^2+2}{N^2}.
\enn
The required error estimate then follows from this since $\wid{G}_N(x)$ is obtained
directly from $\wid{L}_N(x)$. The proof is thus completed.
\end{proof}

By Theorem \ref{thm2} we see that Algorithm \ref{alg1} is of second-order convergence
and that $N$ must increase with the wave number $k$ increasing if the same level of
accuracy is required. Thus, the computational complexity will be increased as $k$ increases
in order to get the same level accuracy.

\begin{remark}\label{re1} {\rm
Algorithm \ref{alg1} is only of second-order accuracy due to the accuracy of FFT.
FFT with higher-order accuracy can be achieved by multiplying the integrand with a weight function.
Then, by removing more terms in \eqref{sing}, we can obtain an algorithm of higher-order convergence.
Note that the second-order accuracy is usually enough for integral equation methods.
}
\end{remark}

\subsection{The three-dimensional case}\label{con-3D}

We now consider the convergence of Algorithm \ref{alg2.3D} in the three-dimensional case.
First, we consider the convergence of the series \eqref{ln}. The following
result is achieved by using the representation of $\hat{L}_{dj}$.

\begin{theorem}\label{h1}
For any $m\in\mathbb{N}^+$, suppose $\mathcal{X},\,\mathcal{Y}_\vep\in C^{m+3}$. For any wave number
$k>0$ and positive integer $N$, $L_N$ converges uniformly to $L$ with $m$-th order accuracy as
$N\rightarrow0$, that is,
\be\label{est}
\|L_{dN}-L_d\|_{L^\infty(D_c)}\le CN^{-m},
\en
where $C$ is a constant depending on $\|\mathcal{X}^{(m+3)}\|_{L_1[-\tilde{c},\tilde{c}]}$,
$\|\mathcal{Y}_\vep^{(n)}\|_{L_1[-\tilde{c},\tilde{c}]}$, $1\le n\le m+3,$ and $k$.
\end{theorem}

\begin{proof}
From $\hat{L}_{dj}=\hat{K}_{dj}-\hat{F}_j$, $\hat{L}_{dj}$ has three terms. Set
\ben
\hat{L}_{dj}&=&\frac{1}{\sqrt{8\pi^2\tilde{c}}}\left[\frac{1}{2\beta_{j_1,j_2}
  (j_3\pi/{\tilde{c}}-\beta_{j_1,j_2})}L_{1,j}
  -\frac{1}{2\beta_{j_1,j_2}(j_3\pi/{\tilde{c}}+\beta_{j_1,j_2})}L_{2,j}\right.\\
 &&\qquad +\left.\frac{1}{(\alpha_1+j_1)^2+(\alpha_2+j_2)^2+(j_3\pi/{\tilde{c}})^2-k^2}L_{3,j}\right],
\enn
where
\ben
L_{1,j}&=&\int_{-\tilde{c}}^{\tilde{c}}\exp[{i(\beta_{j_1,j_2}-j_3\pi/{\tilde{c}})x_3}]\mathcal{X}'(x_3)dx_3;\\
L_{2,j}&=&\int_{-\tilde{c}}^{\tilde{c}}\exp[{-i(\beta_{j_1,j_2}+j_3\pi/{\tilde{c}})x_3}]\mathcal{X}'(-x_3)dx_3;\\
L_{3,j}&=&\int_{D_{\tilde{c}}}\frac{e^{ik|x|}}{4\pi|x|}\big(2ik\mathcal{Y}_\vep'(|x|)+\mathcal{Y}_\vep''(|x|)\big)
 \exp[{-i(\alpha_{1,j_1}x_2+\alpha_{2,j_2}x_2+j_3\pi x_3/{\tilde{c}})}]dx.
\enn
First consider $L_{1,j}$. By integration by parts, we have
\ben
L_{1,j}=\frac{i^{m+2}}{(\beta_{j_1,j_2}-j_3\pi/\tilde{c})^{m+2}}
\int_{-\tilde{c}}^{\tilde{c}}\exp{[i(\beta_{j_1,j_2}-j_3\pi /\tilde{c})x_3]}\mathcal{X}^{(m+3)}(x_3)dx_3.
\enn
Then
\ben
|L_{1,j}|\leqslant\begin{cases}
\frac{1}{|\beta_{j_1,j_2}-j_3\pi/\tilde{c}|^{m+2}}
|| \mathcal{X}^{(m+3)} ||_{L_1[-\tilde{c},\tilde{c}]},\quad\text{ if $\rm{Im}(\beta_{j_1,j_2})$=0,}\\
\frac{e^{-|\beta_{j_1,j_2}|c}}{|\beta_{j_1,j_2}-j_3\pi/\tilde{c}|^{m+2}}
|| \mathcal{X}^{(m+3)} ||_{L_1[-\tilde{c},\tilde{c}]},\quad\text{ if $\rm{Re}(\beta_{j_1,j_2})$=0.}
\end{cases}
\enn
Similarly we can derive the estimate for $L_{2,j}$:
\ben
|L_{2,j}|\leqslant \begin{cases}\frac{1}{|\beta_{j_1,j_2}+j_3\pi/\tilde{c}|^{m+2}}
||\mathcal{X}^{(m+3)}||_{L_1[-\tilde{c},\tilde{c}]},\quad\text{ if $\I(\beta_{j_1,j_2})$=0,}\\
\frac{e^{-|\beta_{j_1,j_2}|c}}{|\beta_{j_1,j_2}+j_3\pi/\tilde{c}|^{m+2}}
||\mathcal{X}^{(m+3)}||_{L_1[-\tilde{c},\tilde{c}]},\quad\text{if $\Rt(\beta_{j_1,j_2})$=0.}
\end{cases}
\enn

Now consider $L_{3,j}$. Define 
$h(x)=\frac{e^{ik|x|}}{4\pi|x|}\big(2ik\mathcal{Y}_\vep'(|x|)
+\mathcal{Y}_\vep''(|x|)\big)e^{-i\alpha_1x_1-i\alpha_2x_2}$.
Then we have 
\ben
|L_{3,j}|\leqslant \frac{1}{|j|^{m+1}}|| h^{(m+1)}||_{L_1(D_{\tilde{c}})}.
\enn
Note that $h^{(m+1)}$ is composed of $\mathcal{Y}^{(n)}$, $1\leqslant n\leqslant m+3$.
Define the set
\begin{equation*}
T_N=\{n\in\mathbb{Z}:\,-N\leqslant n<N \}.
\end{equation*}
For a sufficient large integer $N$, when $j_1,j_2\notin T_{N-1}$, $\beta_{j_1,j_2}$ is a pure imaginary number.
Define
\begin{eqnarray*}
D_{1,N}&=&\{(n_1,n_2,n_3)\in \mathbb{Z}^3:\,n_1,n_2\in T_N,\,n_3\notin T_N\},\\ 
D_{2,N}&=&\{(n_1,n_2,n_3)\in \mathbb{Z}^3:\,\,n_1,n_2\notin T_N, n_3\in\mathbb{Z}\}.
\end{eqnarray*}
Then $\mathbb{Z}^3\setminus S_N=D_{1,N}\bigcup D_{2,N}$.
If $j\in D_{2,N}$, $\beta_{j_1,j_2}$ is a pure imaginary number.
Suppose $|\beta_{j_1.j_2}|$ has a lower bound, that is, $|\beta_{j_1.j_2}|\geqslant b>0$ $\forall j_1,j_2\in\Z$.

With the results above, we can estimate the difference between $L_N$ and L as follows:
\ben
\begin{aligned}
|L_d(x)-L_{dN}(x)|=&\left|\sum_{j\in \mathbb{Z}^3\setminus S_N}\hat{L}_{dj}\phi_j(x)\right|
=\left|\sum_{j\in D_{1,N}}\hat{L}_{dj}\phi_j(x)
  +\sum_{j\in D_{2,N}}\hat{L}_{dj}\phi_j(x)\right|\\
\leqslant &\frac{1}{8\pi^2\tilde{c}}\sum_{j\in D_{1,N}}|\hat{L}_{dj}|
  +\frac{1}{8\pi^2\tilde{c}}\sum_{j\in D_{2,N}}|\hat{L}_{dj}|\\
\leqslant &\frac{1}{8\pi^2\tilde{c}}\sum_{j\in D_{1,N}}
  \left[\frac{|L_{1,j}|}{2b|j_3\pi/\tilde{c}-\beta_{j_1,j_2}|}
+\frac{|L_{2,j}|}{2b|j_3\pi/\tilde{c}+\beta_{j_1,j_2}|}\right]\\
+&\frac{1}{8\pi^2\tilde{c}}\sum_{j\in D_{2,N}}
  \left[\frac{|L_{1,j}|}{2b|j_3\pi/\tilde{c}-\beta_{j_1,j_2}|}
+\frac{|L_{2,j}|}{2b|j_3\pi/\tilde{c}+\beta_{j_1,j_2}|}\right]\\
+&\frac{1}{8\pi^2\tilde{c}}\sum_{j\notin S_N}\frac{|L_{3,j}|}{|(\alpha_1+j_1)^2
  +(\alpha_2+j_2)^2+(j_3\pi/\tilde{c})^2-k^2|}.
\end{aligned}
\enn

Consider the first term. Note that $\forall j_1,j_2\in\mathbb{Z}$, $0\leqslant\Rt(\beta_{j_1,j_2})\leqslant k$.
When $N$ is sufficient large, $|\beta_{j_1,j_2}-j_3\pi/\tilde{c}|\sim |j_3|\pi/\tilde{c}$ for
any $j\in D_{1,N}$. We omit the constants and obtain that 
\ben
\begin{aligned}
\sum_{j\in D_{1,N}}\frac{|L_{1,j}|}{|j_3\pi/\tilde{c}-\beta_{j_1,j_2}|}
\leqslant&\sum_{j\in D_{1,N}}\frac{||\mathcal{X}^{(m+3)}||_{L_1[-\tilde{c},\tilde{c}]}}
{|\beta_{j_1,j_2}-j_3\pi/\tilde{c}|^{m+3}}\sim\sum_{j_1,j_2\in T_N}\sum_{j_3\notin T_N}
\frac{|| \mathcal{X}^{(m+3)}||_{L_1[-\tilde{c},\tilde{c}]}}{|j_3\pi/\tilde{c}|^{m+3}}\\
\leqslant& 4N^2\sum_{|j_3|\geqslant N}\frac{|| \mathcal{X}^{(m+3)}
||_{L_1[-\tilde{c},\tilde{c}]} }{|j_3\pi/\tilde{c}|^{m+3}}\leqslant CN^2\frac{1}{N^{m+2}}=CN^{-m}.
\end{aligned}
\enn
Similarly we can obtain the estimate for the second term:
\ben
\sum_{j\in D_{1,N}}\frac{|L_{2,j}|}{|j_3\pi/\tilde{c}+\beta_{j_1,j_2}|}\leqslant CN^{-m}.
\enn

Consider the third term. For $j\in D_{2,N}$, $\beta_{j_1,j_2}$
is a pure imaginary number and $\I(\beta_{j_1,j_2})=O(|j|)$, so
$|j_3\pi/\tilde{c}-\beta_{j_1,j_2}|\sim \sqrt{(j_3\pi/\tilde{c})^2+|j|^2}$. Then
\ben
\begin{aligned}
\sum_{j\in D_{2,N}}\frac{|L_{1,j}|}{|j_3\pi/\tilde{c}-\beta_{j_1,j_2}|}
\leqslant & \sum_{j\in D_{2,N}}\frac{e^{-|\beta_{j_1,j_2}|c}|| \mathcal{X}^{(m+3)}
||_{L_1[-\tilde{c},\tilde{c}]}}{|\beta_{j_1,j_2}-j_3\pi/\tilde{c}|^{m+3}}
\sim\sum_{j_1,j_2\notin T_N}\sum_{j_3\in\Z}\frac{e^{-|j|c}
||\mathcal{X}^{(m+3)}||_{L_1[-\tilde{c},\tilde{c}]}}{[(j_3\pi/\tilde{c})^2+|j|^2]^\frac{m+3}{2}}\\
\leqslant & C\sum_{|j_1|,|j_2|\geqslant N} \frac{e^{-|j|c}
||\mathcal{X}^{(m+3)}||_{L_1[-\tilde{c},\tilde{c}]}}{N^{m+2}}\leqslant Ce^{-cN}.
\end{aligned}
\enn
A similar estimate can be obtained for the fourth term:
\ben
\sum_{j\in D_{2,N}}\frac{|L_{2,j}|}{|j_3\frac{\pi}{\tilde{c}}+\beta_{j_1,j_2}|}\leqslant Ce^{-cN}.
\enn

The estimate for the last term is trivial, and we will not explain in detail here:
\ben
\sum_{j\in \mathbb{Z}^3\backslash S_N} \frac{|L_{3,j}|}{|(\alpha_1+j_1)^2+(\alpha_2+j_2)^2
+(j_3\frac{\pi}{\tilde{c}})^2-k^2|}\sim \sum_{j\in \mathbb{Z}^3\backslash S_N}
\frac{\parallel h^{(m+1)}\parallel_{L_1(D_{\tilde{c}})}}{|j|^{m+3}}\leqslant CN^{-m}.
\enn

The final estimate of $|L_N(x)-L(x)|$ is then obtained:
\ben
|L_d(x)-L_{dN}(x)|\leqslant (CN^{-m}+CN^{-m}+Ce^{-cN}+Ce^{-cN}+CN^{-m})\leqslant CN^{-m}.
\enn
The proof is completed.
\end{proof}

\begin{theorem}\label{h2}
For any $x\in D_{dc}$, the error between the exact value $G_d(x)$ and the numerical value $\widetilde{G}_{dN}(x)$ 
obtained by Algorithm \ref{alg2.3D} satisfies the following estimate
\ben
|G_d(x)-\widetilde{G}_{dN}(x)|=O\Big(\frac{1}{N^{\sigma}}\Big),
\enn
where $\sigma=\min\{2,m\}$.
\end{theorem}

\begin{proof}
By Theorem \ref{h1} it follows that at any grid point $x^*$,
\ben
|L_d(x^*)-L_{dN}(x^*)|\leqslant CN^{-m}.
\enn
Note that the procedure of computing $L_{dN}(x^*)$ makes use of FFT to compute the Fourier coefficients 
of $L_d$ with the error bounded by $CN^{-2}$. So the value we actually obtained is $\widetilde{L}_{dN}(x^*)$ 
with the Fourier coefficients $\hat{\widetilde{L}}_{dj}$ satisfying that
$|\hat{L}_{dj}-\hat{\widetilde{L}}_{dj}|/|\hat{\widetilde{L}}_{dj}|\leqslant CN^{-2}$. Thus we have
\ben
|L_d(x^*)-\widetilde{L}_{dN}(x^*)|\leqslant|L_d(x^*)-L_{dN}(x^*)|+ |L_{dN}(x^*)-\widetilde{L}_{dN}(x^*)|\leqslant CN^{-m}+CN^{-2}=CN^{-\sigma}.
\enn
For any $x\in D_{dc}$, the value $\widetilde{L}_{dN}(x)$ is obtained by tri-cubic interpolation with 
the interpolation error $CN^{-4}$. Then
\ben
|L_d(x)-\widetilde{L}_{dN}(x)|\leqslant CN^{-\sigma}(1+CN^{-4})=CN^{-\sigma}.
\enn
The required error estimate then follows from this since $\widetilde{G}_{dN}(x)$ is obtained directly from $\widetilde{L}_{dN}(x)$. The proof is completed.
\end{proof}

\begin{remark} {\rm 
From Theorem \ref{h2}, Algorithm \ref{alg2.3D} is of $\sigma$-th order accuracy, where $\sigma$ is the 
smaller one between $m$ and 2. This means that, if $m\geqslant2$, the algorithm is of second order. 
So if we need higher accuracy, we should choose sufficient smooth functions $\mathcal{X},
\mathcal {Y}_\vep$ and use an FFT method with higher-order accuracy.
}
\end{remark}

\subsection{The Maxwells equation case}\label{con-maxwell}

We now  give a convergence analysis and error estimates for Algorithm \ref{alg3.3D} to calculate $\mathbb{G}(x)$. 
The following convergence result of $L^{pq},\,p,q=1,2,3$ follows from Theorem \ref{h1}.

\begin{theorem}\label{m1}
For any integer $m>2$, suppose $\mathcal{X},\,\mathcal{Y}_\vep\in C^{m+3}$. For any wave number $k>0$ 
and positive integer $N$, $L^{pq}_N$ converges uniformly to $L^{pq}$ with $(m-2)$-th order accuracy 
as $N\rightarrow\infty$, that is,
\ben
\|L^{pq}_{dN}-L^{pq}_d\|_{\rm{L}^{\infty}(D_c)}\leqslant CN^{-m+2}.
\enn
\end{theorem}

\begin{proof}
The proof is similar to that of Theorem \ref{h1}, by replacing $\hat{L}_{dj}$ with $\hat{L}^{pq}_{dj}$. 
So we omit it here.
\end{proof}

The following error estimate is also easily obtained from Theorem \ref{h2}, and we omit the proof here.

\begin{theorem}\label{m2}
For any $x\in D_{dc}$, the error between the exact value $\mathbb{G}(x)$ and the numerical value 
$\tilde{\mathbb{G}}_N(x)$ obtained by Algorithm \ref{alg3.3D} satisfies the estimate
\ben
\max_{j,l=1,2,3}|\mathbb{G}_{jl}(x)-\widetilde{\mathbb{G}}_{jl}(x)|=O\Big(\frac{1}{N^{\sigma}}\Big),
\enn
where $\sigma=\min\{2,m-2\}$, $A_{jl}$ stands for the entry of matrix $A$ at the $j$-th row and $l$-th column.
\end{theorem}

\section{Numerical examples}\label{sec4}

In this section, we present several numerical examples of Algorithms \ref{alg1}, \ref{alg2.3D} and \ref{alg3.3D}. 
Our FFT-based method will be compared with some other efficient numerical methods to show that our method is 
competitive when a large number of values are needed. In the numerical methods in this section, the cut off 
functions $\mathcal{X}$ and $\mathcal{Y}_\vep$ are both $C^8$ functions.

\subsection{Numerical examples for $G(x)$}

We first consider the computation of $G(x)$ in 2D. The numerical results obtained by Algorithm \ref{alg1} are presented 
in Subsection \ref{Exm:2D-1}, and the comparison of the FFT-based method with the lattice sums, the Ewald's method 
and the NA method is given in Subsection \ref{Exm:2D-2}. In this section we fix $c=0.6$, $\wid{c}=1$. 
From Remark \ref{remark}, the Fourier coefficients $\hat{F}_{1,j}$ and $\hat{F}_{2,j}$ are assumed to be already known, 
so they could be calculated and saved before any numerical procedures. The calculation was carried out by 2D FFT 
with $2048\times2048$ uniform nodal grids in $D_{\widetilde{c}}$. All computations are carried out with
MATLAB R2012a, on a 2 GHz AMD 3800+ machine with 1 GB RAM.

\subsubsection{Examples of the FFT-based method}\label{Exm:2D-1}

We give four examples in this subsection. For each example, we fix $k$ and $\alpha$, and the Green's function 
is evaluated at four points $P^2_1=(0.01\pi,0)$, $P^2_2=(0.01\pi,0.01)$, $P^2_3=(0.5\pi,0)$ and $P^2_4=(0.5\pi,0.01)$. 
The relative errors at each point for different $N$'s are shown in Tables \ref{table1}-\ref{table4}.

\begin{table}[H]
\centering
\caption{$k=\sqrt{10}$, $\alpha=0.3$}\label{table1}
\begin{tabular}
{|p{1.8cm}<{\centering}|p{1.8cm}<{\centering}|p{1.8cm}<{\centering}
 |p{1.8cm}<{\centering}|p{1.8cm}<{\centering}|}
\hline
 & $P^2_1$ & $P^2_2$ & $P^2_3$ & $P^2_4$\\
\hline
$N=32$&$4.30\rm{E}-04$&$4.60\rm{E}-04$&$5.97\rm{E}-05$&$5.86\rm{E}-05$\\
\hline
$N=64$&$9.12\rm{E}-05$&$1.12\rm{E}-04$&$4.10\rm{E}-06$&$4.06\rm{E}-06$\\
\hline
$N=128$&$4.11\rm{E}-05$&$3.74\rm{E}-05$&$3.48\rm{E}-07$&$3.54\rm{E}-07$\\
\hline
$N=256$&$2.87\rm{E}-07$&$6.82\rm{E}-07$&$4.60\rm{E}-07$&$4.62\rm{E}-07$\\
\hline
$N=512$&$4.08\rm{E}-07$&$3.48\rm{E}-07$&$4.57\rm{E}-07$&$4.58\rm{E}-07$\\
\hline
$N=1024$&$1.70\rm{E}-07$&$1.66\rm{E}-07$&$4.57\rm{E}-07$&$4.58\rm{E}-07$\\
\hline
\end{tabular}
\end{table}

\begin{table}[H]
\centering
\caption{$k=5$, $\alpha=0.3$}\label{table2}
\begin{tabular}
{|p{1.8cm}<{\centering}|p{1.8cm}<{\centering}|p{1.8cm}<{\centering}
 |p{1.8cm}<{\centering}|p{1.8cm}<{\centering}|}
\hline
 & $P^2_1$ & $P^2_2$ & $P^2_3$ & $P^2_4$\\
\hline
$N=32$&$1.63\rm{E}-03$&$1.72\rm{E}-03$&$7.63\rm{E}-05$&$7.63\rm{E}-05$\\
\hline
$N=64$&$2.67\rm{E}-04$&$3.31\rm{E}-04$&$1.52\rm{E}-06$&$1.46\rm{E}-06$\\
\hline
$N=128$&$1.22\rm{E}-04$&$1.11\rm{E}-04$&$5.98\rm{E}-07$&$6.19\rm{E}-07$\\
\hline
$N=256$&$6.16\rm{E}-07$&$1.82\rm{E}-06$&$6.90\rm{E}-07$&$6.97\rm{E}-07$\\
\hline
$N=512$&$9.61\rm{E}-07$&$7.93\rm{E}-07$&$6.95\rm{E}-07$&$6.95\rm{E}-07$\\
\hline
$N=1024$&$2.59\rm{E}-07$&$2.55\rm{E}-07$&$6.95\rm{E}-07$&$6.95\rm{E}-07$\\
\hline
\end{tabular}
\end{table}

\begin{table}[H]
\centering
\caption{$k=50$, $\alpha=\sqrt{2}$}\label{table3}
\begin{tabular}
{|p{1.8cm}<{\centering}|p{1.8cm}<{\centering}|p{1.8cm}<{\centering}
 |p{1.8cm}<{\centering}|p{1.8cm}<{\centering}|}
\hline
 & $P^2_1$ & $P^2_2$ & $P^2_3$ & $P^2_4$\\
\hline
$N=128$&$3.96\rm{E}-02$&$3.66\rm{E}-02$&$2.28\rm{E}-04$&$1.84\rm{E}-04$\\
\hline
$N=256$&$1.81\rm{E}-03$&$2.07\rm{E}-03$&$1.53\rm{E}-05$&$9.71\rm{E}-06$\\
\hline
$N=512$&$2.45\rm{E}-04$&$2.30\rm{E}-04$&$6.84\rm{E}-06$&$6.64\rm{E}-06$\\
\hline
$N=1024$&$3.20\rm{E}-05$&$3.41\rm{E}-05$&$6.62\rm{E}-06$&$6.62\rm{E}-06$\\
\hline
\end{tabular}
\end{table}

\begin{table}[H]
\centering
\caption{$k=100$, $\alpha=-\sqrt{2}$}\label{table4}
\begin{tabular}
{|p{1.8cm}<{\centering}|p{1.8cm}<{\centering}|p{1.8cm}<{\centering}
 |p{1.8cm}<{\centering}|p{1.8cm}<{\centering}|}
\hline
 & $P^2_1$ & $P^2_2$ & $P^2_3$ & $P^2_4$\\
\hline
$N=256$&$3.89\rm{E}-02$&$3.75\rm{E}-02$&$1.44\rm{E}-04$&$6.40\rm{E}-05$\\
\hline
$N=512$&$2.76\rm{E}-03$&$2.82\rm{E}-03$&$4.75\rm{E}-06$&$8.76\rm{E}-06$\\
\hline
$N=1024$&$4.72\rm{E}-04$&$4.26\rm{E}-04$&$8.95\rm{E}-06$&$9.37\rm{E}-06$\\
\hline
\end{tabular}
\end{table}

From the examples above, it is known that Algorithm \ref{alg1} is convergent. When $N$ is relatively small, 
the relative errors decay at the rate shown in Subsection \ref{con-2D}. When $N$ gets larger, the relative 
errors do not decay as expected, as the 2D FFT algorithm to compute the Fourier coefficients $\hat{F}_{1,j}$ 
and $\hat{F}_{2,j}$ has an error level of $\frac{1}{2048^2}\sim 10^{-7}$.

The preparation time increases as $N$ increases, and is the most time-consuming part in the numerical procedure. 
The time cost in the preparation step is around $3.5$ seconds for $N=32,64,128,256$, $4.5$ seconds for $N=512$ 
and $7.5$ seconds for $N=1024$, on average. So the preparation time increase as $N$ increases, 
but it increases slowly and does not cost too much time. Moreover, the preparation step only needs to run once 
at the beginning of the whole numerical scheme.

On the other hand, the calculation time does not depend on $N$. In the examples above, it takes about 
$8.5\times 10^{-5}$ seconds on average. Thus the calculation step is very efficient, especially 
when a large number of values are required.

\subsubsection{Comparison with other methods}\label{Exm:2D-2}

We now compare the FFT-based method (FM in tables) with the lattice sums (LS in tables), 
the Ewald's method (EM in tables) and the NA$_1$ method from \cite{KR}.

First, four groups of examples are given to compare the FFT-based method with the lattice sums and 
the Ewald's method. In each group, we calculate the Green's function with a fixed $k$ and $\alpha$ 
at the four points $P^2_1$, $P^2_2$, $P^2_3$ and $P^2_4$. Similar to the FFT-based method,
the lattice sums also have two steps: the preparation step and the calculation step. 
As the preparation step takes no more than $1$ minute, we omit the preparation time. 
We adjust the parameters in each method to achieve a similar accuracy. If any of the methods fails 
to achieve the required accuracy, we will choose the parameters so that the method could achieve 
the best accuracy. The relative errors and calculation times are presented in Tables \ref{table5}-\ref{table8}.

\begin{table}[H]
\centering
\caption{$k=5$, $\alpha=0.3$}\label{table5}
\begin{tabular}
{|p{0.5cm}<{\centering}
 |p{1.2cm}<{\centering}|p{1.5cm}<{\centering}
 |p{1.2cm}<{\centering}|p{1.5cm}<{\centering}
 |p{1.2cm}<{\centering}|p{1.5cm}<{\centering}
 |p{1.2cm}<{\centering}|p{1.5cm}<{\centering}|}
\hline
\multirow{2}*{}&\multicolumn{2}{|c|}{$P^2_1$}&\multicolumn{2}{|c|}{$P^2_2$}&\multicolumn{2}{|c|}{$P^2_3$}
&\multicolumn{2}{|c|}{$P^2_4$}\\
\cline{2-9}
             &r-error& c-time(s)&r-error&c-time(s)&r-error&c-time(s)&r-error& c-time(s)\\
\hline
LS&$3\rm{E}-07$&$8.6\rm{E}-4$&$6\rm{E}-07$&$8.6\rm{E}-4$&$4\rm{E}-07$&$3.1\rm{E}-3$&$4\rm{E}-06$&$1.4\rm{E}-3$ \\
\hline
EM&$4\rm{E}-07$&$1.9\rm{E}-3$ &$4\rm{E}-07$&$2.0\rm{E}-3$&$3\rm{E}-07$&$1.9\rm{E}-3$ &$3\rm{E}-07$&$1.9\rm{E}-3$ \\
\hline
FM&$9\rm{E}-07$&$9.9\rm{E}-5$& $8\rm{E}-07$&$9.2\rm{E}-5$ &$7\rm{E}-07$&$9.2\rm{E}-5$& $7\rm{E}-07$&$8.8\rm{E}-5$ \\
\hline
\end{tabular}
\end{table}

\begin{table}[H]
\centering
\caption{$k=50$, $\alpha=\sqrt{2}$}\label{table6}
\begin{tabular}
{|p{0.5cm}<{\centering}
 |p{1.2cm}<{\centering}|p{1.5cm}<{\centering}
 |p{1.2cm}<{\centering}|p{1.5cm}<{\centering}
 |p{1.2cm}<{\centering}|p{1.5cm}<{\centering}
 |p{1.2cm}<{\centering}|p{1.5cm}<{\centering}|}
\hline
\multirow{2}*{}&\multicolumn{2}{|c|}{$P^2_1$}&\multicolumn{2}{|c|}{$P^2_2$}&\multicolumn{2}{|c|}{$P^2_3$}
&\multicolumn{2}{|c|}{$P^2_4$}\\
\cline{2-9}
             &r-error& c-time(s)&r-error&c-time(s)&r-error&c-time(s)&r-error& c-time(s)\\
\hline
LS&$5\rm{E}-05$&$9.0\rm{E}-4$&$2\rm{E}-05$&$1.1\rm{E}-3$&$2\rm{E}-01$&$1.4\rm{E}-3$&$2\rm{E}-01$&$1.4\rm{E}-3$ \\
\hline
EM&$2\rm{E}-05$&$3.0\rm{E}-3$ &$2\rm{E}-05$&$3.7\rm{E}-3$&$8\rm{E}-06$&$3.0\rm{E}-3$ &$8\rm{E}-06$&$3.0\rm{E}-3$ \\
\hline
FM&$3\rm{E}-05$&$8.6\rm{E}-5$& $3\rm{E}-05$&$9.0\rm{E}-5$ &$7\rm{E}-06$&$8.7\rm{E}-5$& $7\rm{E}-06$&$9.0\rm{E}-5$ \\
\hline
\end{tabular}
\end{table}

\begin{table}[H]
\centering
\caption{$k=100$, $\alpha=0.5$}\label{table7}
\begin{tabular}
{|p{0.5cm}<{\centering}
 |p{1.2cm}<{\centering}|p{1.5cm}<{\centering}
 |p{1.2cm}<{\centering}|p{1.5cm}<{\centering}
 |p{1.2cm}<{\centering}|p{1.5cm}<{\centering}
 |p{1.2cm}<{\centering}|p{1.5cm}<{\centering}|}
\hline
\multirow{2}*{}&\multicolumn{2}{|c|}{$P^2_1$}&\multicolumn{2}{|c|}{$P^2_2$}&\multicolumn{2}{|c|}{$P^2_3$}
&\multicolumn{2}{|c|}{$P^2_4$}\\
\cline{2-9}
             &r-error& c-time(s)&r-error&c-time(s)&r-error&c-time(s)&r-error& c-time(s)\\
\hline
LS&$1\rm{E}-04$&$1.2\rm{E}-3$&$1\rm{E}-04$&$1.2\rm{E}-3$&$9\rm{E}-01$&$1.3\rm{E}-3$&$9\rm{E}-01$&$1.2\rm{E}-3$ \\
\hline
EM&$2\rm{E}-04$&$3.4\rm{E}-3$ &$2\rm{E}-04$&$3.4\rm{E}-3$&$9\rm{E}-06$&$4.7\rm{E}-3$ &$9\rm{E}-06$&$3.6\rm{E}-3$ \\
\hline
FM&$4\rm{E}-04$&$9.2\rm{E}-5$& $4\rm{E}-04$&$8.8\rm{E}-5$ &$1\rm{E}-05$&$8.5\rm{E}-5$& $1\rm{E}-05$&$9.2\rm{E}-5$ \\
\hline
\end{tabular}
\end{table}

\begin{table}[H]
\centering
\caption{$k=200$, $\alpha=0.8$}\label{table8}
\begin{tabular}
{|p{0.5cm}<{\centering}
 |p{1.2cm}<{\centering}|p{1.5cm}<{\centering}
 |p{1.2cm}<{\centering}|p{1.5cm}<{\centering}
 |p{1.2cm}<{\centering}|p{1.5cm}<{\centering}
 |p{1.2cm}<{\centering}|p{1.5cm}<{\centering}|}
\hline
\multirow{2}*{}&\multicolumn{2}{|c|}{$P^2_1$}&\multicolumn{2}{|c|}{$P^2_2$}&\multicolumn{2}{|c|}{$P^2_3$}
&\multicolumn{2}{|c|}{$P^2_4$}\\
\cline{2-9}
             &r-error& c-time(s)&r-error&c-time(s)&r-error&c-time(s)&r-error& c-time(s)\\
\hline
LS&$1\rm{E}-02$&$1.4\rm{E}-3$&$1\rm{E}-02$&$1.3\rm{E}-3$&$6\rm{E}-01$&$1.2\rm{E}-3$&$6\rm{E}-01$&$1.0\rm{E}-3$ \\
\hline
EM&$2\rm{E}-03$&$4.4\rm{E}-3$ &$2\rm{E}-03$&$4.6\rm{E}-3$&$1\rm{E}-05$&$4.5\rm{E}-3$ &$1\rm{E}-05$&$5.3\rm{E}-3$ \\
\hline
FM&$4\rm{E}-03$&$8.9\rm{E}-5$& $4\rm{E}-03$&$8.5\rm{E}-5$ &$7\rm{E}-06$&$8.3\rm{E}-5$& $8\rm{E}-06$&$8.3\rm{E}-5$ \\
\hline
\end{tabular}
\end{table}

From the four examples it is seen that, when $k$ gets larger, the lattice sums need more time to reach 
the required accuracy or even fail to work when the point is not too close to $0.$ 
When $k$ and $r$ are both small enough, the lattice sums method is faster than the Ewald's method, 
but is much slower than the FFT-based method. On the other hand, the Ewald's method can reach any accuracy 
if required, but its speed is not so competitive. The FFT-based method is the fastest one during 
the calculation step, despite that it takes about $4$ seconds in the first example ($N=512$), 
and about $7.5$ seconds in the second to forth examples ($N=1024$) to do the preparation before 
the calculation step. This leads to the conclusion that, when $k$ is small enough, the FFT-based method 
is more competitive, while, when $k$ is larger, the FFT-based method wins if a large amount of values 
are needed, but if the number of evaluations is small or a very high accuracy is required, 
the Ewald's method is a better choice.

We also compare our method with the so-called NA method given in \cite{KR}, where the authors 
introduced three methods, NA$_1$, NA$_2$ and NA$_3$, to calculate the quasi-periodic Green's functions 
for very high wave numbers, i.e., $k\geq 10^4$. We only compare the FFT-based method
with NA$_1$ here since the three methods perform similarly when $k$ is not very high.
The examples are taken when $k=10.2,100.2,200.2$, at the points $(0,0.01)$ and $(0,0.1)$.
Similar to the examples before, we choose proper parameters such that both of the two methods 
can achieve similar relative errors.

\begin{table}[H]
\centering
\caption{Comparision with NA$_1$}\label{table9}
\begin{tabular}{|p{0.8cm}<{\centering}|p{1.8cm}<{\centering}|p{1.8cm}<{\centering}|p{1.8cm}<{\centering}
|p{1.4cm}<{\centering}|p{1.8cm}<{\centering}|p{1.4cm}<{\centering}|}
\hline
\multirow{2}*{k}&\multirow{2}*{point}&\multicolumn{2}{|c|}{$NA_1$}&\multicolumn{2}{|c|}{$FM$}\\
\cline{3-6}
      &       &r-error& c-time(s)&r-error&c-time(s)\\
\hline
10.2&$(0,0.01)$&1E-08&  0.045 &7E-09&0.000087\\
\hline
10.2&$(0,0.1)$&1E-07&  0.045 &3E-07&0.000089\\
\hline
100.2&$(0,0.01)$&1E-05&  0.046 &1E-05&0.000087\\
\hline
100.2&$(0,0.1)$&7E-06&  0.043 &4E-06&0.000084\\
\hline
200.2&$(0,0.01)$&6E-05&  0.046 &8E-05&0.000083\\
\hline
200.2&$(0,0.1)$&1E-05&  0.040 &6E-05&0.000086\\
\hline
\end{tabular}
\end{table}

From Table \ref{table9} it is seen that the FFT-based method is much faster than NA$_1$. But NA$_1$ also
has its advantage that it can reach a much higher rate of accuracy without costing more time.
It also works well for very large $k$'s for which the FFT-based method takes up too much time and memory 
to achieve the same accuracy. When the wave number $k$ is not that large, e.g., $k\sim 100$, 
the FFT-based method is more competitive if a large number of values are needed.

\subsection{Numerical examples for $G_d(x)$}\label{sec8.3D}
\setcounter{equation}{0}

We now present some numerical examples for the calculation of $G_d(x)$ in 3D by Algorithm \ref{alg2.3D}.
We present the examples to show the results from our methods in Subsection \ref{sec8.3D.1} 
and the comparison of our method with the lattice sums and the Ewald method in Subsection \ref{sec8.3D.2}. 
In this section we fix $c=0.6$, $\tilde{c}=1$.
Our numerical procedure in this section is implemented with Fortran 11.1.064, on a 2.40GHz NF560D2 
machine with 96 GB RAM.

\subsubsection{Examples of the FFT-based method}\label{sec8.3D.1}

We give the examples using the FFT-based method to compute the 3D quasi-periodic Green's function $G_d(x)$. 
We present six groups of examples. In each group, the wave number $k$ is fixed, and the values at four points
$P^3_1=(0,1.5,0.0008)$, $P^3_2=(0.03,0.03,0.0008)$, $P^3_3=(0,1.5,0.1)$ and $P^3_4=(0.03,0.03,0.1)$ are 
evaluated for different $N$'s. The eigenfunction expansion \eqref{he} with a sufficiently large number of terms 
is used to compute the "exact value" of $G_d(x)$.
The relative errors of the values calculated at each point with different $N$'s are shown 
in Tables \ref{table8.1}-\ref{table8.6}.

\begin{table}[H]
\centering
\caption{$k=1$, $\alpha_1=0.1$, $\alpha_2=0.2$}\label{table8.1}
\begin{tabular}
{|p{1.8cm}<{\centering}|p{1.8cm}<{\centering}|p{1.8cm}<{\centering}
 |p{1.8cm}<{\centering}|p{1.8cm}<{\centering}|}
\hline
 & $P^3_1$ & $P^3_2$ & $P^3_3$ & $P^3_4$\\
\hline
$N=32$&$3.73\rm{E}-04$&$6.97\rm{E}-05$&$3.94\rm{E}-04$&$1.72\rm{E}-04$\\
\hline
$N=64$&$1.18\rm{E}-07$&$1.24\rm{E}-07$&$1.20\rm{E}-07$&$1.88\rm{E}-07$\\
\hline
$N=128$&$3.59\rm{E}-07$&$7.29\rm{E}-09$&$3.54\rm{E}-07$&$1.88\rm{E}-08$\\
\hline
$N=256$&$5.48\rm{E}-08$&$2.52\rm{E}-10$&$5.39\rm{E}-08$&$6.53\rm{E}-10$\\
\hline
$N=512$&$1.38\rm{E}-09$&$9.67\rm{E}-12$&$1.32\rm{E}-09$&$2.49\rm{E}-10$\\
\hline
\end{tabular}
\end{table}

\begin{table}[H]
\centering
\caption{$k=5$, $\alpha_1=0.1$, $\alpha_2=0.2$}\label{table8.2}
\begin{tabular}
{|p{1.8cm}<{\centering}|p{1.8cm}<{\centering}
  |p{1.8cm}<{\centering}|p{1.8cm}<{\centering}|p{1.8cm}<{\centering}|}
\hline
 & $P^3_1$ & $P^3_2$ & $P^3_3$ & $P^3_4$\\
\hline
$N=32$&$7.76\rm{E}-04$&$1.20\rm{E}-04$&$6.81\rm{E}-04$&$2.89\rm{E}-04$\\
\hline
$N=64$&$6.44\rm{E}-05$&$4.58\rm{E}-06$&$6.43\rm{E}-05$&$1.17\rm{E}-05$\\
\hline
$N=128$&$1.01\rm{E}-05$&$9.31\rm{E}-07$&$9.83\rm{E}-06$&$2.42\rm{E}-06$\\
\hline
$N=256$&$1.55\rm{E}-06$&$3.23\rm{E}-08$&$1.51\rm{E}-06$&$8.39\rm{E}-08$\\
\hline
$N=512$&$4.58\rm{E}-08$&$1.17\rm{E}-08$&$4.50\rm{E}-08$&$3.05\rm{E}-08$\\
\hline
\end{tabular}
\end{table}

\begin{table}[H]
\centering
\caption{$k=10$, $\alpha_1=0.8$, $\alpha_2=\sqrt{2}$}\label{table8.3}
\begin{tabular}
{|p{1.8cm}<{\centering}|p{1.8cm}<{\centering}|p{1.8cm}<{\centering}
  |p{1.8cm}<{\centering}|p{1.8cm}<{\centering}|}
\hline
 & $P^3_1$ & $P^3_2$ & $P^3_3$ & $P^3_4$\\
\hline
$N=32$&$3.49\rm{E}-02$&$1.22\rm{E}-03$&$3.33\rm{E}-02$&$3.16\rm{E}-03$\\
\hline
$N=64$&$1.99\rm{E}-03$&$5.17\rm{E}-05$&$1.99\rm{E}-03$&$1.36\rm{E}-04$\\
\hline
$N=128$&$4.15\rm{E}-04$&$3.76\rm{E}-06$&$4.22\rm{E}-04$&$9.86\rm{E}-06$\\
\hline
$N=256$&$6.29\rm{E}-05$&$1.81\rm{E}-07$&$6.40\rm{E}-05$&$4.74\rm{E}-07$\\
\hline
$N=512$&$1.45\rm{E}-06$&$5.16\rm{E}-08$&$1.48\rm{E}-06$&$1.35\rm{E}-07$\\
\hline
\end{tabular}
\end{table}

\begin{table}[H]
\centering
\caption{$k=25$, $\alpha_1=0.8$, $\alpha_2=\sqrt{2}$}\label{table8.4}
\begin{tabular}
{|p{1.8cm}<{\centering}|p{1.8cm}<{\centering}
  |p{1.8cm}<{\centering}|p{1.8cm}<{\centering}|p{1.8cm}<{\centering}|}
\hline
 & $P^3_1$ & $P^3_2$ & $P^3_3$ & $P^3_4$\\
\hline
$N=64$&$3.97\rm{E}-02$&$1.48\rm{E}-03$&$3.94\rm{E}-02$&$3.63\rm{E}-03$\\
\hline
$N=128$&$3.35\rm{E}-03$&$6.04\rm{E}-05$&$3.78\rm{E}-03$&$1.49\rm{E}-04$\\
\hline
$N=256$&$4.99\rm{E}-04$&$6.19\rm{E}-06$&$5.65\rm{E}-04$&$1.53\rm{E}-05$\\
\hline
$N=512$&$7.86\rm{E}-06$&$6.53\rm{E}-07$&$9.55\rm{E}-06$&$1.60\rm{E}-06$\\
\hline
\end{tabular}
\end{table}

\begin{table}[H]
\centering
\caption{$k=50$, $\alpha_1=\sqrt{3}$, $\alpha_2=0.5$}\label{table8.5}
\begin{tabular}
{|p{1.8cm}<{\centering}|p{1.8cm}<{\centering}
  |p{1.8cm}<{\centering}|p{1.8cm}<{\centering}|p{1.8cm}<{\centering}|}
\hline
 & $P^3_1$ & $P^3_2$ & $P^3_3$ & $P^3_4$\\
\hline
$N=128$&$3.09\rm{E}-02$&$2.26\rm{E}-04$&$2.74\rm{E}-02$&$5.48\rm{E}-04$\\
\hline
$N=256$&$5.05\rm{E}-03$&$2.29\rm{E}-05$&$4.51\rm{E}-03$&$5.76\rm{E}-05$\\
\hline
$N=512$&$2.52\rm{E}-04$&$1.46\rm{E}-06$&$2.34\rm{E}-04$&$3.04\rm{E}-06$\\
\hline
$N=736$&$1.03\rm{E}-04$&$1.72\rm{E}-07$&$9.26\rm{E}-05$&$3.67\rm{E}-07$\\
\hline
\end{tabular}
\end{table}

\begin{table}[H]
\centering
\caption{$k=100$, $\alpha_1=\sqrt{3}$, $\alpha_2=0.5$}\label{table8.6}
\begin{tabular}
{|p{1.8cm}<{\centering}|p{1.8cm}<{\centering}|p{1.8cm}<{\centering}
 |p{1.8cm}<{\centering}|p{1.8cm}<{\centering}|}
\hline
 & $P^3_1$ & $P^3_2$ & $P^3_3$ & $P^3_4$\\
\hline
$N=128$&$7.80\rm{E}-02$&$5.48\rm{E}-03$&$1.05\rm{E}-01$&$1.55\rm{E}-02$\\
\hline
$N=256$&$9.68\rm{E}-03$&$1.13\rm{E}-03$&$1.56\rm{E}-02$&$3.14\rm{E}-03$\\
\hline
$N=512$&$1.15\rm{E}-03$&$4.01\rm{E}-05$&$5.74\rm{E}-04$&$1.02\rm{E}-04$\\
\hline
$N=736$&$1.31\rm{E}-04$&$3.64\rm{E}-06$&$1.94\rm{E}-04$&$9.27\rm{E}-06$\\
\hline
\end{tabular}
\end{table}

These examples above show that Algorithm \ref{alg2.3D} is convergent with the relative errors bounded 
by that estimated in Section \ref{con-3D}. Although the convergence of the algorithm is quite fast, 
as shown in Section \ref{con-3D}, when $k$ is large, $N$ should be larger to achieve an acceptable accuracy.

Similar to the 2D case, the preparation step is the most time-consuming part in the numerical scheme, 
and the time taken up by this step depends on $N$. Figure \ref{fig3D.1} gives the relationship between
the average preparation time and $N$.

\begin{figure}[H]
\centering\includegraphics[height=6cm,width=8cm]{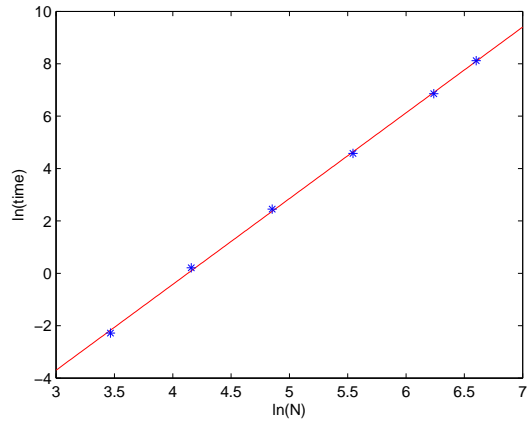}
\caption{preparation time}\label{fig3D.1}
\end{figure}

In Figure \ref{fig3D.1}, the blue stars are the data and the red line is the linear fitting of the data.
Figure \ref{fig3D.1} shows that the preparation time increases at the rate of $N^{3.3}$ as $N$ increases. 
It increases significantly when $N$ gets larger.

Fortunately, the preparation step needs to be run only once before the calculation step starts. 
The calculation step takes about $1.3\times 10^{-6}$ seconds on average, which does not depend on 
any parameter. This is a quite efficient step, which is very suitable for a large amount of evaluations.

\subsubsection{Comparison with other methods}\label{sec8.3D.2}

In this subsection, we compare the FFT-based method (FM) with the Ewald method (EM) and the lattice sums (LS).
We will apply these three methods to six examples with different $k$'s at the points $P^3_1$ and $P^3_2$. 
Different from the 2D case, the preparation time is no longer too small to be ignored in the 3D case.
In Tables \ref{table8.7}-\ref{table8.12}, we list the relative errors and the
preparation time (p-time) and calculation times (c-time) in the case when $k=1,5,10,25,50,100$.
The parameter $N$ in the FFT-based method is changed when $k$ is different.

\begin{table}[H]
\centering
\caption{$k=1$, $N=32$}\label{table8.7}
\begin{tabular}
{|p{0.8cm}<{\centering}|p{1.8cm}<{\centering}|p{1.6cm}<{\centering}
 |p{1.8cm}<{\centering}|p{1.6cm}<{\centering}|p{1.6cm}<{\centering}|}
\hline
\multirow{2}*{}&\multicolumn{2}{|c|}{$P^3_1$}&\multicolumn{2}{|c|}{$P^3_2$}&\multirow{2}*{p-time(s)}\\
\cline{2-5}
             &r-error& c-time(s)&r-error&c-time(s)&\\
\hline
EM&$3.60\rm{E}-05$&$1.00\rm{E}-3$&$1.02\rm{E}-06$&$1.00\rm{E}-3$&0 \\
\hline
LS&$4.29\rm{E}-02$&$6.90\rm{E}-5$ &$2.39\rm{E}-05$&$6.90\rm{E}-5$&13.4 \\
\hline
FM&$3.73\rm{E}-04$&$1.30\rm{E}-6$& $6.97\rm{E}-05$&$1.29\rm{E}-6$ &0.10 \\
\hline
\end{tabular}
\end{table}

\begin{table}[H]
\centering
\caption{$k=5$, $N=64$}\label{table8.8}
\begin{tabular}
{|p{0.8cm}<{\centering}|p{1.8cm}<{\centering}|p{1.6cm}<{\centering}
 |p{1.8cm}<{\centering}|p{1.6cm}<{\centering}|p{1.6cm}<{\centering}|}
\hline
\multirow{2}*{}&\multicolumn{2}{|c|}{$P^3_1$}&\multicolumn{2}{|c|}{$P^3_2$}&\multirow{2}*{p-time(s)}\\
\cline{2-5}
             &r-error& c-time(s)&r-error&c-time(s)&\\
\hline
EM&$4.09\rm{E}-05$&$1.40\rm{E}-2$&$3.38\rm{E}-06$&$1.40\rm{E}-2$&0 \\
\hline
LS&$2.81\rm{E}-03$&$6.90\rm{E}-5$&$2.99\rm{E}-06$&$6.80\rm{E}-5$&14 \\
\hline
FM&$6.44\rm{E}-05$&$1.28\rm{E}-6$& $4.58\rm{E}-06$&$1.29\rm{E}-6$ &1.24 \\
\hline
\end{tabular}
\end{table}

\begin{table}[H]
\centering
\caption{$k=10$, $N=128$}\label{table8.9}
\begin{tabular}
{|p{0.8cm}<{\centering}|p{1.8cm}<{\centering}|p{1.6cm}<{\centering}
 |p{1.8cm}<{\centering}|p{1.6cm}<{\centering}|p{1.6cm}<{\centering}|}
\hline
\multirow{2}*{}&\multicolumn{2}{|c|}{$P^3_1$}&\multicolumn{2}{|c|}{$P^3_2$}&\multirow{2}*{p-time(s)}\\
\cline{2-5}
             &r-error& c-time(s)&r-error&c-time(s)&\\
\hline
EM&$8.01\rm{E}-04$&$5.20\rm{E}-2$&$3.19\rm{E}-05$&$4.90\rm{E}-2$&0 \\
\hline
LS&$1.00\rm{E}-04$&$1.80\rm{E}-4$&$1.91\rm{E}-05$&$1.82\rm{E}-4$&34 \\
\hline
FM&$4.15\rm{E}-04$&$1.30\rm{E}-6$& $3.76\rm{E}-06$&$1.30\rm{E}-6$ &11.42 \\
\hline
\end{tabular}
\end{table}

\begin{table}[H]
\centering
\caption{$k=25$, $N=256$}\label{table8.10}
\begin{tabular}
{|p{0.8cm}<{\centering}|p{1.8cm}<{\centering}|p{1.6cm}<{\centering}
 |p{1.8cm}<{\centering}|p{1.6cm}<{\centering}|p{1.6cm}<{\centering}|}
\hline
\multirow{2}*{}&\multicolumn{2}{|c|}{$P^3_1$}&\multicolumn{2}{|c|}{$P^3_2$}&\multirow{2}*{p-time(s)}\\
\cline{2-5}
             &r-error& c-time(s)&r-error&c-time(s)&\\
\hline
EM&$2.26\rm{E}-04$&$0.33$&$1.31\rm{E}-05$&$0.32$&0 \\
\hline
LS&$2.89\rm{E}-04$&$6.65\rm{E}-4$&$2.34\rm{E}-04$&$6.68\rm{E}-4$&125 \\
\hline
FM&$4.99\rm{E}-04$&$1.28\rm{E}-6$& $6.19\rm{E}-06$&$1.30\rm{E}-6$ &96.2 \\
\hline
\end{tabular}
\end{table}

\begin{table}[H]
\centering
\caption{$k=50$, $N=512$}\label{table8.11}
\begin{tabular}
{|p{0.8cm}<{\centering}|p{1.8cm}<{\centering}|p{1.6cm}<{\centering}
 |p{1.8cm}<{\centering}|p{1.6cm}<{\centering}|p{1.6cm}<{\centering}|}
\hline
\multirow{2}*{}&\multicolumn{2}{|c|}{$P^3_1$}&\multicolumn{2}{|c|}{$P^3_2$}&\multirow{2}*{p-time(s)}\\
\cline{2-5}
             &r-error& c-time(s)&r-error&c-time(s)&\\
\hline
EM&$1.75\rm{E}-04$&$1.22$&$5.96\rm{E}-05$&$1.21$&0 \\
\hline
LS&$4.15\rm{E}-02$&$2.33\rm{E}-3$&$6.68\rm{E}-04$&$2.34\rm{E}-3$&811 \\
\hline
FM&$2.52\rm{E}-04$&$1.31\rm{E}-6$& $1.46\rm{E}-06$&$1.28\rm{E}-6$ &940 \\
\hline
\end{tabular}
\end{table}

\begin{table}[H]
\centering
\caption{$k=100$, $N=736$}\label{table8.12}
\begin{tabular}
{|p{0.8cm}<{\centering}|p{1.8cm}<{\centering}|p{1.6cm}<{\centering}
 |p{1.8cm}<{\centering}|p{1.6cm}<{\centering}|p{1.6cm}<{\centering}|}
\hline
\multirow{2}*{}&\multicolumn{2}{|c|}{$P^3_1$}&\multicolumn{2}{|c|}{$P^3_2$}&\multirow{2}*{p-time(s)}\\
\cline{2-5}
             &r-error& c-time(s)&r-error&c-time(s)&\\
\hline
EM&$2.59\rm{E}-04$&$5.32$&$2.75\rm{E}-06$&$6.44$&0 \\
\hline
LS&$7.79\rm{E}-01$&$1.62\rm{E}-3$&$6.23\rm{E}-03$&$1.63\rm{E}-3$&828 \\
\hline
FM&$1.31\rm{E}-04$&$1.29\rm{E}-6$& $3.64\rm{E}-06$&$1.30\rm{E}-6$ &3322 \\
\hline
\end{tabular}
\end{table}

From the tables above, we see that the Ewald's method is able to achieve any required accuracy 
and does not have any preparation time. However, the calculation time taken up by this method grows significantly 
as $k$ increases. Similar to the 2D case, the lattice sums also suffer from the problems of convergence
when $k$ is relatively larger. The calculation time of this method is much less than that of the Ewald method,
but still more than that of the FFT-based method. The FFT-based method is
the fastest one and works well for all $k$'s with a dependent calculation time. However, when $k$ increases, 
a larger $N$ is required for the same accuracy, which leads to a significant growth in the preparation time 
and memory. Thus, when a small number of values or high accuracies are needed,
the Ewald's method is the best choice, while, if a large amount of evaluations is required, 
the FFT-based method is more competitive.

\subsection{Numerical examples for $\mathbb{G}(x)$}\label{sec10.3D}
\setcounter{equation}{0}

In this section, we present some numerical examples for the calculation of $\mathbb{G}(x)$ by Algorithm \ref{alg3.3D}.
We will present six numerical examples to show the efficiency of this algorithm. The maximum relative error
($\max_{j,l=1,2,3}\frac{|\mathbb{G}_{jl}(x)-\tilde{\mathbb{G}}_{jl}(x)|}{|\mathbb{G}_{jl}(x)|}$),
is calculated at each point, and the values for $N=32,64,128,256,512$ are shown 
in Tables \ref{table10.1}-\ref{table10.6}.

\begin{table}[H]
\centering
\caption{$k=1$, $\alpha_1=0.1$, $\alpha_2=0.2$}\label{table10.1}
\begin{tabular}
{|p{1.8cm}<{\centering}|p{1.8cm}<{\centering}|p{1.8cm}<{\centering}
|p{1.8cm}<{\centering}|p{1.8cm}<{\centering}|}
\hline
 & $P^3_1$ & $P^3_2$ & $P^3_3$ & $P^3_4$\\
\hline
$N=32$&$1.09\rm{E}-02$&$2.65\rm{E}-06$&$3.42\rm{E}-01$&$4.18\rm{E}-04$\\
\hline
$N=64$&$9.00\rm{E}-06$&$6.96\rm{E}-08$&$2.49\rm{E}-04$&$8.82\rm{E}-06$\\
\hline
$N=128$&$2.50\rm{E}-06$&$3.04\rm{E}-10$&$6.99\rm{E}-06$&$1050\rm{E}-08$\\
\hline
$N=256$&$5.61\rm{E}-07$&$1.12\rm{E}-11$&$3.61\rm{E}-07$&$3.55\rm{E}-11$\\
\hline
$N=512$&$4.71\rm{E}-07$&$5.14\rm{E}-12$&$9.90\rm{E}-09$&$1.48\rm{E}-11$\\
\hline
\end{tabular}
\end{table}

\begin{table}[H]
\centering
\caption{$k=5$, $\alpha_1=0.1$, $\alpha_2=0.2$}\label{table10.2}
\begin{tabular}
{|p{1.8cm}<{\centering}|p{1.8cm}<{\centering}|p{1.8cm}<{\centering}
 |p{1.8cm}<{\centering}|p{1.8cm}<{\centering}|}
\hline
 & $P^3_1$ & $P^3_2$ & $P^3_3$ & $P^3_4$\\
\hline
$N=32$&$6.92\rm{E}-03$&$6.25\rm{E}-05$&$4.12\rm{E}-02$&$3.86\rm{E}-04$\\
\hline
$N=64$&$8.84\rm{E}-05$&$2.39\rm{E}-06$&$8.85\rm{E}-05$&$5.56\rm{E}-06$\\
\hline
$N=128$&$2.49\rm{E}-05$&$4.93\rm{E}-07$&$2.42\rm{E}-05$&$1.16\rm{E}-06$\\
\hline
$N=256$&$3.81\rm{E}-06$&$1.71\rm{E}-08$&$3.68\rm{E}-06$&$4.07\rm{E}-08$\\
\hline
$N=512$&$1.26\rm{E}-07$&$6.19\rm{E}-09$&$9.24\rm{E}-08$&$1.44\rm{E}-08$\\
\hline
\end{tabular}
\end{table}

\begin{table}[H]
\centering
\caption{$k=10$, $\alpha_1=0.8$, $\alpha_2=\sqrt{2}$}\label{table10.3}
\begin{tabular}
{|p{1.8cm}<{\centering}|p{1.8cm}<{\centering}|p{1.8cm}<{\centering}
 |p{1.8cm}<{\centering}|p{1.8cm}<{\centering}|}
\hline
 & $P^3_1$ & $P^3_2$ & $P^3_3$ & $P^3_4$\\
\hline
$N=32$&$1.48\rm{E}-01$&$1.00\rm{E}-03$&$1.36\rm{E}-01$&$2.74\rm{E}-03$\\
\hline
$N=64$&$2.12\rm{E}-03$&$4.24\rm{E}-05$&$2.24\rm{E}-03$&$1.31\rm{E}-04$\\
\hline
$N=128$&$4.41\rm{E}-04$&$3.09\rm{E}-06$&$4.56\rm{E}-04$&$1.16\rm{E}-05$\\
\hline
$N=256$&$6.68\rm{E}-05$&$1.48\rm{E}-07$&$6.91\rm{E}-05$&$4.63\rm{E}-07$\\
\hline
$N=512$&$1.53\rm{E}-06$&$4.23\rm{E}-08$&$1.60\rm{E}-06$&$1.59\rm{E}-07$\\
\hline
\end{tabular}
\end{table}

\begin{table}[H]
\centering
\caption{$k=25$, $\alpha_1=0.8$, $\alpha_2=\sqrt{2}$}\label{table10.4}
\begin{tabular}
{|p{1.8cm}<{\centering}|p{1.8cm}<{\centering}|p{1.8cm}<{\centering}
 |p{1.8cm}<{\centering}|p{1.8cm}<{\centering}|}
\hline
 & $P^3_1$ & $P^3_2$ & $P^3_3$ & $P^3_4$\\
\hline
$N=64$&$4.53\rm{E}-02$&$1.51\rm{E}-03$&$4.47\rm{E}-02$&$6.22\rm{E}-03$\\
\hline
$N=128$&$4.62\rm{E}-03$&$6.64\rm{E}-05$&$5.12\rm{E}-03$&$2.54\rm{E}-04$\\
\hline
$N=256$&$6.93\rm{E}-04$&$6.34\rm{E}-06$&$7.70\rm{E}-04$&$2.62\rm{E}-05$\\
\hline
$N=512$&$1.32\rm{E}-05$&$6.67\rm{E}-07$&$1.52\rm{E}-05$&$2.74\rm{E}-06$\\
\hline
\end{tabular}
\end{table}

\begin{table}[H]
\centering
\caption{$k=50$, $\alpha_1=\sqrt{3}$, $\alpha_2=0.5$}\label{table10.5}
\begin{tabular}
{|p{1.8cm}<{\centering}|p{1.8cm}<{\centering}|p{1.8cm}<{\centering}
 |p{1.8cm}<{\centering}|p{1.8cm}<{\centering}|}
\hline
 & $P^3_1$ & $P^3_2$ & $P^3_3$ & $P^3_4$\\
\hline
$N=128$&$3.20\rm{E}-02$&$7.35\rm{E}-04$&$2.78\rm{E}-02$&$1.69\rm{E}-03$\\
\hline
$N=256$&$5.28\rm{E}-03$&$5.93\rm{E}-05$&$4.55\rm{E}-03$&$1.65\rm{E}-04$\\
\hline
$N=512$&$2.69\rm{E}-04$&$6.80\rm{E}-06$&$2.43\rm{E}-04$&$1.21\rm{E}-05$\\
\hline
$N=736$&$1.09\rm{E}-04$&$8.00\rm{E}-07$&$9.37\rm{E}-05$&$1.43\rm{E}-06$\\
\hline
\end{tabular}
\end{table}

\begin{table}[H]
\centering
\caption{$k=100$, $\alpha_1=\sqrt{3}$, $\alpha_2=0.5$}\label{table10.6}
\begin{tabular}
{|p{1.8cm}<{\centering}|p{1.8cm}<{\centering}|p{1.8cm}<{\centering}
 |p{1.8cm}<{\centering}|p{1.8cm}<{\centering}|}
\hline
 & $P^3_1$ & $P^3_2$ & $P^3_3$ & $P^3_4$\\
\hline
$N=256$&$2.00\rm{E}-02$&$1.90\rm{E}-03$&$4.68\rm{E}-02$&$1.46\rm{E}-02$\\
\hline
$N=512$&$2.22\rm{E}-03$&$4.13\rm{E}-05$&$1.54\rm{E}-03$&$4.73\rm{E}-04$\\
\hline
$N=736$&$2.35\rm{E}-04$&$4.69\rm{E}-06$&$5.88\rm{E}-04$&$4.30\rm{E}-05$\\
\hline
\end{tabular}
\end{table}

From the examples above, similar results can be concluded as those in Subsection \ref{sec8.3D.1}.
We do not show the preparation and calculation time as they are both about six times longer
than those taken by Algorithm \ref{alg2.3D}. This implies that our method for the calculation
of $\mathbb{G}(x)$ is still highly efficient, especially when a large number of values are needed.

\section*{Acknowledgements}

This work was partly supported by the NNSF of China grants 91430102 and 91630309.
We thank the referees for their constructive comments which improved this paper.


\begin{thebibliography}{99}

\bibitem{ASSL} T. Arens, K. Sandfort, S. Schmitt and A. Lechleiter 2013 Analysing Ewald's method
for the evaluation of Green's functions for periodic media, {\em IMA J. Numer. Anal. \bf78}(3) 405-431.

\bibitem{BKM} G. Beylkin, C. Kurcz and L. Monz\'{o}n 2008 Fast algorithms for Helmholtz Green's functions, 
{\em Proc. R. Soc. \bf A464} 3301-3326.

\bibitem{BD} O.P. Bruno and B. Delourme 2014 Rapidly convergent two-dimensional quasi-periodic Green 
function throughout the spectrum - including Wood anomalies, {\em J. Comput. Phys. \bf262} 262-290.

\bibitem{BF} O.P. Bruno and A.G. Fernandez-Lado 2017 Rapidly convergent quasi-periodic Green functions 
for scattering by arrays of cylinders--including Wood anomalies, {\em Proc. R. Soc. \bf A473}, in press.

\bibitem{CWW} F. Capolino, D.R. Wilton, and W.A. JohnsonZ 2002 Efficient computation of
the 2-D Green's function for 1-D periodic structures using the Ewald method,
{\em IEEE Trans. Antennas Propagat. \bf53}(4) 2977-2984.

\bibitem{TJ}T.F. Eibert, J.L. Volakis, D.R. Wilton, and D.R. Jackson 1999 Hybrid FE/BI modeling
of 3-D doubly periodic structures utilizing triangular prismatic elements and an MPIE formulation
accelerated by the Ewald transformation, {\em IEEE Trans. Antennas Propagat. \bf47} 843-850.

\bibitem{GET}N. Guerin, S. Enoch, and G. Tayeb 2001 Combined method for the computation
of the doubly periodic Green's function, {\em J. Electrom. Waves Appl. \bf15} 205-221.

\bibitem{TH} T. Hohage 2006 Fast numerical solution of the electromagnetic medium scattering
problem and applications to the inverse problem, {\em J. Comput. Phys. \bf 214} 224-238.

\bibitem{JM}R.E. Jorgenson and R. Mittra 1990 Efficient calculation of the free-space periodic
Green's function, {\em IEEE Trans. Antennas Propagat. \bf 38} 633-642.

\bibitem{KR} H. Kurkcu and F. Reitich 2009 Stable and efficient evaluation of periodized
Green's functions for the Helmholtz equation at high frequencies,
{\em J. Comput. Phys. \bf228} 75-95.

\bibitem{KM}A. Kustepeli and A.Q. Martin 2000 On the splitting parameter in the Ewald method,
{\em IEEE Microwave Guided Wave Lett. \bf10} 168-170.

\bibitem{LN} A. Lechleiter and D.L. Nguyen 2012 Spectral volumetric integral equation methods for  
acoustic medium scattering in 3D waveguide, {\em IMA J. Numer. Anal. \bf32}(3) 813-844.

\bibitem{LM}F. Lekien and J. Marsden 2005 Tricubic interpolation in three dimensions.
{\em Int. J. Numer. Meth. Engng. \bf63} 455-471.

\bibitem{CML1} C.M. Linton 1998 The Green's function for the two-dimensional Helmholtz equation in
periodic domains, {\em J. Eng. Math. \bf33} 377-402.

\bibitem{LT} C.M. Linton and I. Thompson 2009 One- and two-dimensional lattice sums
for the threedimensional Helmholtz equation, {\em J. Comput. Phys. \bf228} 1815-1829.

\bibitem{CML2} C.M. Linton 2010 Lattice Sums for the Helmholtz Equation,
{\em SIAM Rev. \bf52}(4) 630-674.

\bibitem{CML3} C.M. Linton 2015 Two-dimensional, phase modulated lattice sums with application to 
the Helmholtz Green's function, {\em J. Math. Phys. \bf56} 013505.

\bibitem{MP} A.W. Mathis and A.F. Peterson 1998 Efficient electromagnetic analysis of a doubly
infinite array of rectangular apertures, {\em IEEE Trans. Microwave Theory Tech. \bf46} 46-54.

\bibitem{AMO}A. Moroz 2002 On the computation of the free-space doubly-periodic Green's
function of the threedimensional Helmholtz equation, {\em J. of Electrom. Waves Appl. \bf16} 457-465.

\bibitem{AM} A. Moroz 2006 Quasi-periodic Green's functions of the Helmholtz and Laplace equations,
{\em J. Phys. \bf A39} 11247-11282.

\bibitem{OJW}S. Oroskar, D.R. Jackson, and D.R. Wilton 2006 Efficient computation of the 2D periodic
Green's function using the Ewald method, {\em J. Comput. Phys. \bf219} 899-911.

\bibitem{OC} N.A. Ozdemir and C. Craeye 2009 Evaluation of the periodic Green's function near Wood's anomaly 
and application to the array scanning method, {\em IEEE Antennas Propag. Soc. Int. Symposium}.

\bibitem{VGP} V.G. Papanicolaou 1999 Ewald's method revisited: Rapidly convergent
series representations of certain Green's functions, {\em J. Comput. Anal. Appl. \bf1} 105-114.

\bibitem{SK} K. Sandfort, {\em The Factorization Method for Inverse Scattering From
Periodic Inhomogeneous Media}, PhD Thesis, KIT, KIT Scientific Publishing, 2010, Germany.

\bibitem{SF} M.G. Silveirinha and C.A. Fernandes 2005 A new acceleration technique
with exponential convergence rate to evaluate periodic Green functions,
{\em IEEE Trans. Antennas Propagat. \bf53} 347-355.

\bibitem{SCBBM}I. Stevanovi'c, P. Crespo-Valero, K. Blagovi'c, F. Bongard, and J.R. Mosig 2006
Integral equation analysis of 3-D metallic objects arranged in 2-D lattices using
the Ewald transformation, {\em IEEE Trans. Microwave Theory Tech. \bf54} 3688-3697.

\bibitem{VG1} G. Vainikko, Fast solvers of the Lippmann-Schwinger equation,
in: {\em Direct and Inverse Problems of Mathematical Physics} (eds. R.P. Gilbert,
J. Kajiwara and Y. Xu), Kluwer, Dordrecht, The Netherlands, 2000, pp. 423-440.

\bibitem{VBB} G. Valerio, P. Baccarelli, P. Burghignoli, and A. Galli 2007
Comparative analysis of acceleration techniques for 2-D and 3-D Green's functions
in periodic structures along one and two directions, {\em IEEE Trans. Antennas Propagat. \bf55} 1630-1643.

\bibitem{JACW}J.A.C. Weideman 1994 Computation of the complex error function,
{\em SIAM. J. Numer Anal \bf31}(5) 1497-1518.

\bibitem{YY} K. Yasumoto and K. Yoshitomi 1999 Efficient calculation of lattice
sums for free-space periodic Green's function, {\em IEEE Trans. Antennas Propagat. \bf47} 1050-1055.

\end{thebibliography}
\end{document}